\documentclass[a4paper,10pt]{amsart}

\usepackage{amssymb,amsmath,amsthm,mathrsfs,enumerate,graphicx, color}
\usepackage[pdfpagelabels,colorlinks,linkcolor=blue,citecolor=black,urlcolor=blue]{hyperref}
\usepackage{esint}
\newtheorem{thm}{Theorem}[section]

\newtheorem{cor}[thm]{Corollary}
\newtheorem{lem}[thm]{Lemma}
\newtheorem{prop}[thm]{Proposition}
\newtheorem{defn}[thm]{Definition}
\newtheorem{rem}[thm]{Remark}

\newcommand{\Mu}{\mathbb{M}_1(\mu)}

\newcommand{\lesi}{\lesssim}

\newcommand{\pd}{p(\cdot)}
\newcommand{\px}{p(x)}

\newcommand{\f}{\frac}

\newcommand{\om}{\omega}
\newcommand{\Om}{\Omega}

\newcommand{\vc}{\infty}
\newcommand{\Rn}{\mathbb{R}^n}

\newcommand{\di}{{\rm div}}

\title[Regularity estimates for elliptic equations with non-standard
growth]{Weighted variable exponent Sobolev estimates for elliptic equations with non-standard
	growth and measure data}         
\author{The Anh Bui}
\address{Department of Mathematics, Macquarie University, NSW 2109,
Australia}
\email{the.bui@mq.edu.au, bt\_anh80@yahoo.com}
 
\author{Xuan Thinh Duong}
\address{Department of Mathematics, Macquarie University, NSW 2109,
Australia}

\email{xuan.duong@mq.edu.au}


\begin{document}

\begin{abstract}
Consider the following nonlinear elliptic equation of $p(x)$-Laplacian type with nonstandard growth 
\begin{equation*}
\left\{
\begin{aligned}
&\di  a(Du, x)=\mu \quad &\text{in}& \quad \Om,\\
&u=0  \quad &\text{on}& \quad \partial\Om,
\end{aligned}
\right.
\end{equation*}
where $\Om$ is a Reifenberg domain in $\Rn$, $\mu$ is a Radon measure defined on $\Om$ with finite total mass and the nonlinearity $a: \mathbb{R}^n\times \mathbb{R}^n\to \mathbb{R}^n$ is modeled upon the $p(\cdot)$-Laplacian.

We prove the estimates on weighted {\it variable exponent} Lebesgue spaces for gradients of  solutions to this equation 
in terms of Muckenhoupt--Wheeden type estimates. As a consequence, we obtain some  
new results such as the weighted $L^q-L^r$ regularity (with constants $q  < r$) and estimates on Morrey spaces for gradients of the solutions to this non-linear equation.
\end{abstract}
\date{}

\maketitle

\tableofcontents

\section{Introduction}
Partial differential equations including nonlinear elliptic and parabolic problems with nonstandard growth conditions 
have recently been studied extensively by many mathematicians as these equations have had a wide 
range of applications in many fields such as mathematical physics,  elastic mechanics, image processing and electro-rheological fluid dynamics. See for example \cite{AM1, AM2, BOR, DR,CLR,Hetal, RR, R, Ru, ZZ, Z} and the references therein.
 
In this paper we consider the following nonlinear elliptic equation of $p(x)$-Laplacian type with nonstandard growth 

\begin{equation}\label{Maineq}
\left\{
\begin{aligned}
&\di  a(Du, x)=\mu \quad &\text{in}& \quad \Om,\\
&u=0  \quad &\text{on}& \quad \partial\Om,
\end{aligned}
\right.
\end{equation}
where $\Om$ is a bounded open domain in $\Rn$ and $\mu$ is a Radon measure defined on $\Om$ with finite total mass. The nonlinearity $a: \mathbb{R}^n\times \mathbb{R}^n\to \mathbb{R}^n$ is modeled upon the $p(\cdot)$-Laplacian.

Recently a systematic study on nonlinear elliptic of $p(x)$-Laplacian of type \eqref{Maineq} with measure data has been received a lot of attention. In the particular case of $p$-Laplacian type equations (i.e. $p$ is independent of $x$), the existence results for the solutions to nonlinear elliptic and parabolic equations with measure data were proved in \cite{BG,BG2,Be.etal,BGO}. Then the regularity results for solutions of those equations were obtained in \cite{Ph} for the elliptic case and in \cite{QHN} for the parabolic case.  For the general case of  $p(x)$-Laplacian type equation, some interesting results regarding to entropy solutions and very weak solutions were obtained in \cite{BWZ,AHHL,SU,ZY}.\\

Recall that a weak solution to the problem \eqref{Maineq} is a function $u\in W^{1,\pd}_0(\Om)$ such that
$$
\int_{\Om }a(Du, x)\cdot D\varphi dx = \int_\Om \varphi d\mu \ \ \ \ \text{for all $\varphi\in C_0^\vc(\Om)$}.
$$
See Section 2 for definition of the variable exponent Sobolev space $W^{1,\pd}_0(\Om)$.\\

For $x\in \mathbb{R}^n$, we define
$$
\Mu(x)=\sup_{r>0}\sup_{B_r(y)\ni x}\f{|\mu|(B_r(y))}{r^{n-1}}\sim \sup_{r>0}\f{|\mu|(B_r(x))}{r^{n-1}}
$$
to be the first order fractional maximal function associated to the measure $\mu$, 
where $B_r(z):=\{y: |z-y|<r\}$ is the open ball with center $z\in \mathbb{R}^n$ and radius $r$. 
It is not difficult to see that for a nonnegative locally finite measure $\nu$ in $\mathbb{R}^n$, 
the maximal function $\mathbb{M}_1(\nu)$ is dominated by the Riesz potential related to $\nu$. 
More precisely, we have
$$
\mathbb{M}_1(\nu)(x)\leq c_n \mathbf{I}_1(\nu)(x):=c_n\int_{\mathbb{R}^n}\f{d\nu(y)}{|x-y|^{n-1}}, x\in \mathbb{R}^n.
$$

We note that the problem of getting estimates for the solution via fractional maximal functions and  nonlinear potentials is an interesting topic and has attracted a great deal of attention in recent years. We now list some of the papers related to this direction.\\
\begin{enumerate}[{\rm (i)}]
	\item The first two results appeared in \cite{KM1, KM2}, where the authors proved a pointwise potential estimate for solutions to the quasi elliptic equation via Wolff potentials. Later, in \cite{TW}, by using a different approach, the authors extended this result to obtain the pointwise estimates for solutions
	to non-homogeneous quasi-linear equations of $p$-Laplacian type with measure data in terms of Wolff type nonlinear potentials.\\

\item In the series of works by Mingione and his collaborators, they extended the results in \cite{TW} to the pointwise estimates for the gradient of solutions, instead of the solution, via nonlinear potentials. More precisely, the pointwise estimate for the gradient of solutions to the degenerate quasilinear equations of $p$-Laplacian type was first proved in \cite{M1} for the case $p=2$. The case $p\neq 2$ can be found in \cite{DM1, DM2, DM3, KMi1, KMi2}.\\

\item The gradient estimates for solutions to the equation \eqref{Maineq} in terms of variable exponent potentials were obtained in \cite{BH, BaH} corresponding to $p(\cdot)\geq 2$ and $p(\cdot)>2-1/n$ by using Mingione and Duzzar's approach. Then, optimal integrability results for solutions of the $\px$-Laplace equation in variable exponent weak Lebesgue spaces were obtained in \cite{AHHL}.\\

\item The regularity results for the solutions to the nonlinear elliptic equation of $p(x)$-Laplacian type of the form
\[
\left\{
\begin{aligned}
&\di a(Du, x)=\di(|F|^{\pd -2}F) \quad &\text{in}& \quad \Om,\\
&u=0  \quad &\text{on}& \quad \partial\Om,
\end{aligned}
\right.
\]
 were proved in \cite{BOR, BO} in the scale of Lebesgue and generalized Lebesgue spaces, respectively.\\
 
 \item In \cite{Ph}, the author proved weighted estimates for gradients of solution to the equation \eqref{Maineq} in the particular case of $p$-Laplacian type via the maximal operator $\mathbb{M}_1$.\\

\end{enumerate}

The main aim of this paper is to prove the weighted $L^{q(\cdot)}$ estimates for  gradients of the solutions to the equation \eqref{Maineq} via the fractional maximal function $\mathbb{M}_1$. These estimates are similar to those in \cite{Ph} as in (v) above but we obtain the estimates for  solutions of nonlinear elliptic equations of $p(x)$-Laplacian type and in terms of the weighted $L^{q(\cdot)}$. See Theorem \ref{mainthm1} and its subsequence results in Corollaries \ref{mainthm2}, \ref{mainthm3} and \ref{mainthm4}.\\

We now give some comments on the technical ingredients used in this paper. 
In order to prove the main results, we employ the maximal function technique  which makes use of the variant of Vitali covering lemma 
and good $\lambda$-inequality. This technique was originated in \cite{CP} and was used in various settings. 
See for example \cite{BW, Ph,MP2,QHN, BDL}. However, some major modifications need to be carried out since 
the maximal function techniques are not applicable directly to our problem due to the presence of 
the variable exponent $\px$ which rules out the homogeneity of the equation \eqref{Maineq}. To overcome this problem, we  make use of the log-H\"older condition of the exponent functions and  some subtle localized estimates. \\

 The organization of the paper is as follows. In Setion 2, we set up the assumptions on the nonlinearity $a$ and the underlying domain $\Om$, and then state the main results. See Theorem \ref{mainthm1} and its subsequence results such as  Corollaries \ref{mainthm2}, \ref{mainthm3}, \ref{mainthm4} and Theorem \ref{mainthm5}. In Section 3, we prove some interior and boundary comparison estimates which play an important role in the sequel. The proofs of the main results are given in Section 4.\\

 Throughout the paper, we always use $C$ and $c$ to denote positive constants that are independent of the main parameters involved but whose values may differ from line to line. We write
 $A\lesi B$ if there is a universal constant $C$ so that $A\leq CB$ and $A\sim B$ if $A\lesi B$ and $B\lesi A$. For $a, b\in \mathbb{R}$ we denote $a\wedge b =\min\{a,b\}$. We also denote  by  $\mathcal{O}(\texttt{data})$ the infinitely small quantity with respect to the \texttt{data}, i.e., $\lim_{\texttt{data}\to 0}\mathcal{O}(\texttt{data})=0$.

\section{Assumptions  and Statement of the results}

We will begin with some notations which will be used frequently in the sequel.
\begin{itemize}
	
	\item For $x\in \mathbb{R}^n$ and $r>0$, we denote by $B_r(x):=\{y\in \mathbb{R}^n: |x-y|<r\}$ the open ball with center $x$ 
	and radius $r$ in $\mathbb{R}^{n}$.
	
	\item We also denote $\Om_r(x)=\Om\cap B_r(x)$ and $\partial_w \Om_r(x)=\partial \Om\cap B_r(x)$. If $x$ is the origin, we simply write $B_r$, $\Om_r$ and $\partial_w\Om_r$ for $B_r(x)$, $\Om_r(x)$ and $\partial_w\Om_r(x)$, respectively.
	
	\item For a measurable function $f$ on a measurable subset $E\subset \mathbb{R}^{n}$ we define
	$$
	\overline{f}_E =\fint_E f dx=\f{1}{|E|}\int_E f dx.
	$$
\end{itemize}

We now recall  some definitions and basic properties concerning the variable exponent Lebesgue spaces in \cite{CF}. 
Let $\Omega$ be a subset of $\mathbb{R}^n$. For $p(\cdot) : \Om \to (0,\vc)$, we define the variable exponent  Lebesgue spaces $L^{p(\cdot)}(\Om)$ to be a generalization of the classical Lebesgue
spaces consisting of all measurable functions on $\Om$ satisfying
$$
\int_\Om |f(x)|^{p(x)} dx<\vc,
$$
with the norm
$$
\|f\|_{L^{p(\cdot)}(\Om)}=\inf\Big\{\lambda>0: \int_{\Omega}\left(\f{|f(x)|}{\lambda}\right)^{p(x)}dx\leq 1\Big\}.
$$
It is well-known that
$$
\|f\|_{L^{p(\cdot)}(\Om)}\leq 1 \Longleftrightarrow \int_\Om |f(x)|^{p(x)} dx\leq 1,
$$
and  if $1\leq p(x)<\vc$ for all $x\in \Om$ then $\|\cdot\|_{L^{p(\cdot)}(\Om)}$ is a norm and hence $L^{p(\cdot)}(\Om)$ is a Banach space. In general, $\|\cdot\|_{L^{p(\cdot)}(\Om)}$  is a quasi-norm.

The generalized Sobolev space $W^{1,\pd}(\Om)$ is defined as the set of all measurable functions $f\in L^{\pd}(\Om)$ whose derivative $Df\in L^{\pd}(\Om)$. If $f\in W^{1,\pd}(\Om)$, then its norm is defined by
$$
\|f\|_{W^{1,\pd}(\Om)}=\|f\|_{L^{\pd}(\Om)}+\|\,|Df|\,\|_{L^{\pd}(\Om)}.
$$ 
The space $W^{1,\pd}_0(\Om)$ is defined as a closure of $C^\vc_c(\Om)$ in $W^{1,\pd}(\Om)$. The generalized Lebesgue--Sobolev spaces play an important role in studying regularity estimates for elliptic and parabolic problems. See for example \cite{DR, Dea, CF} and the references therein for further discussions.

\subsection{Our assumptions}

In this paper, we assume that the nonlinearity $a(\xi,x): \mathbb{R}^n\times \mathbb{R}^n \to \mathbb{R}^n$  is  measurable in $x$ for every $\xi\in \Rn$ and differentiable in $\xi$ for each $x\in \Rn$. In addition,  there exist the variable exponent $\pd: \Om\to (1,\vc)$ and  constants $\Lambda_1\geq  \Lambda_2>0$, $s\in [0,1]$  so that
\begin{equation}\label{eq1-functiona}
(s^2+|\xi|^2)^{1/2}|D_\xi a(\xi,x)|+|a(\xi,x)|\leq \Lambda_1(s^2+|\xi|^2)^{\f{p(x)-1}{2}},
\end{equation}
and
\begin{equation}\label{eq1s-functiona}
\langle D_\xi a(\xi,x)\eta,\eta\rangle  \geq \Lambda_2 (s^2+|\xi|^2)^{\f{p(x)-2}{2}}|\eta|^2,
\end{equation}
for every $x, \xi, \eta \in \mathbb{R}^n$.

Note that these two conditions imply that
\begin{equation}\label{eq2-functiona}
 \langle a(\xi,x)-a(\eta,x),\xi-\eta\rangle  \geq \Lambda_2 (s^2+|\xi|^2+|\eta|^2)^{\f{p(x)-2}{2}}|\xi-\eta|^2\geq \Lambda_2|\xi-\eta|^{\px}
\end{equation}
for every $x,\xi,\eta \in \mathbb{R}^n$.

Moreover, the exponent function $\pd: \Om\to (1,\vc)$ is assumed to be continuous, satisfies the  bounds
$$
2-\f{1}{n}< \gamma_1\leq p(x)\leq \gamma_2<\vc,
$$
and the log-H\"older continuity condition
\begin{equation}
\label{cond1-px}|p(x)-p(y)|\leq \omega(|x-y|), \ \forall x,y\in \Rn
\end{equation}
where $\omega:[0,\vc)\to [0,\vc)$ is a non-decreasing function satisfying
\begin{equation}
\label{cond2-px} \lim_{r\to 0^+}\omega(r)\log\Big(\f{1}{r}\Big)=0.
\end{equation}

We choose a number $R_\om$ so that for all $0<r<R_\omega$,
\begin{equation}
\label{eq-R1}
0<\om(r)\log\Big(\f{1}{r}\Big)\leq \f{1}{2}.
\end{equation}   

We set
$$
\Theta(a,B_r(y))(x)=\sup_{\xi\in \Rn}\left|\f{a(\xi,x)}{(s^2+|\xi|^2)^{\f{\px-1}{2}}}-\overline{\left(\f{a(\xi,\cdot)}{(s^2+|\xi|^2)^{\f{\pd-1}{2}}}\right)}_{B_r(y)}\right|
$$
which is used in the next definition concerning the nonlinearity $a$.
\begin{defn}
	Let $R_0,\delta>0$. The nonlinearity $a$ is said to satisfy a small $(\delta, R_0)$-BMO condition if 
	\begin{equation}\label{eq2-Assumption}
	[{\bf a}]_{2,R_0}:=\sup_{y\in \mathbb{R}^n}\sup_{0<r\leq R_0}\,\fint_{B_{r}(y)}|\Theta(a,B_r(y))(x)|^2dx\leq \delta^2.
	\end{equation}
\end{defn}
\begin{rem}
	\label{rem1}
	This condition was introduced in \cite{BOR}. Note that if \eqref{eq2-Assumption} holds true, then for any $\beta\in [1,\vc)$ we have
	$$
	[{\bf a}]_{\beta,R_0}:=\sup_{y\in \mathbb{R}^n}\sup_{0<r\leq R_0}\,\fint_{B_{r}(y)}|\Theta(a,B_r(y))(x)|^\beta dx\leq \mathcal{O}(\delta).
	$$
\end{rem}
\medskip

\subsection{Reifenberg flat domains}
Concerning  the underlying domain $\Omega$, we do not assume any smoothness condition on $\Omega$, but the following  flatness condition.
\begin{defn}\label{defn2}
	Let $\delta, R_0>0$. The domain $\Om$ is said to be a $(\delta, R_0)$ Reifenberg flat domain if for every $x\in \partial\Om$ and $0<r\leq R_0$, then there exists a coordinate system depending on $x$ and $r$, whose variables are denoted by $y=(y_1,\dots,y_n)$ such that in this new coordinate system $x$ is the origin and
	\begin{equation}\label{eq1-Assumption}
	B_{r}\cap\{y: y_n>\delta r\}\subset B_{r}\cap \Omega \subset \{y: y_n>-\delta r\}.
	\end{equation}
\end{defn}

\begin{rem}\label{rem1}
	{\rm (a)} The condition of  $(\delta, R_0)$-Reifenberg flatness condition was first introduced in \cite{R}. This condition does not require any smoothness on the boundary of $\Om$, but
	sufficiently flat in the Reifenberg's sense. The Reifenberg flat domain includes domains with rough boundaries of fractal
	nature, and Lipschitz domains with small Lipschitz constants.
	For further discussions about the Reifenberg domain, we refer to \cite{R, DT, Toro} and the references therein. \\

	{\rm (b)} If $\Om$ is a $(\delta, R_0)$ Reifenberg domain, then for any $x_0\in \partial \Om$ and $0<\rho<R_0(1-\delta)$ there exists a coordinate system, whose variables are denoted by $y=(y_1,\ldots, y_n)$ such that in this coordinate system the origin is an interior point of $\Om$, $x_0=(0,\ldots, 0, -\f{\delta \rho}{1-\delta})$ and
	$$
	B_{\rho}^+\subset B_{\rho}\cap \Om\subset B_{\rho}\cap  \left\{y: y_n>-\f{2\delta \rho}{1-\delta}\right\}.
	$$

	{\rm (c)} For $x\in \Om$ and $0<r<R_0$, we have
	\begin{equation}
	\label{eq1-Reifenberg domain}
	\f{|B_r(x)|}{|B_r(x)\cap \Om|}\leq \Big(\f{2}{1-\delta}\Big)^n.
	\end{equation}
\end{rem}

\bigskip

{\it Throughout the paper, we always assume that the domain $\Om$ is a $(\delta, R_0)$ Reifenberg flat domain, and the nonlinearity $a$ satisfies \eqref{eq1-functiona}, \eqref{eq1s-functiona} and the small $(\delta, R_0)$-BMO condition \eqref{eq2-Assumption}.}
\subsection{Statement of the results}
Let $1\leq p<\infty$. A nonnegative locally integrable function $w$ belongs to the {\sl Muckenhoupt class} $A_p$, say $w\in A_p$, if there exists a positive constant $C$ so that
$$[w]_{A_p}:=\Big(\fint_B w(x)dx\Big)\Big(\fint_Q w^{-1/(p-1)}(x)dx\Big)^{p-1}\leq C, \quad\mbox{if}\; 1<p<\infty,$$
and
$$
\fint_B w(x)dx\leq C \mathop{\mbox{ess-inf}}\limits_{x\in B}w(x),\quad{\rm if}\; p=1,
$$
for all balls $B$ in $\mathbb R^n$. We say that $w\in A_\infty$ if $w\in A_p$ for some $p\in [1,\infty)$. We shall denote $w(E) :=\int_E w(x)dx$ for any measurable set $E \subset \mathbb{R}^n$.

For a weight $w$ and $0<q<\vc$ we define
$$
L^q_w(\Om)=\Big\{f: \|f\|_{L^q_w(\Om)}:=\Big(\int_{\Om }|f(x)|^q w(x)dx\Big)^{1/q}<\vc\Big\}.
$$

We now record the following property of the Muckenhoupt weights in \cite{Du}.
\begin{lem}\label{weightedlemma2}
	Let  $w\in A_\vc$. Then, there exist $\kappa_w\in (0,1)$, and a constant $c_w>1$ such that for any ball $B$ and any measurable subset $E\subset B$,
	$$w(E) \leq c_w \Big(\f{|E|}{|B|}\Big)^{\kappa_w}w(B).
	$$
\end{lem}

We now consider the continuous exponent function $q(\cdot): \Om\to (0,\vc)$ satisfying the log-H\"older continuity condition:
\begin{equation}
\label{cond1-qx}|q(x)-q(y)|\leq \nu(|x-y|), \ \forall x,y\in \Om,
\end{equation}
where $\nu:[0,\vc)\to [0,\vc)$ is a non-decreasing function satisfying
\begin{equation}
\label{cond2-qx} \lim_{r\to 0^+}\nu(r)\log\Big(\f{1}{r}\Big)=0.
\end{equation}
We also assume that there exist constants $\gamma_3$ and $\gamma_4$ such that
\begin{equation}
\label{cond3-qx}
0<\gamma_3\leq q(x)\leq \gamma_4<\vc, \ \ \ \forall x\in \Om.
\end{equation}

Our first main result gives the weighted $L^{q(\cdot)}$ regularity for the solutions to problem \eqref{Maineq}.
\begin{thm}\label{mainthm1}
	Let $q(\cdot)$ be defined as in \eqref{cond1-qx}, \eqref{cond2-qx} and \eqref{cond3-qx}, $w\in A_{\vc}$ and $0<\sigma_0<\min\left\{\f{n(\gamma_1-1)}{n-1},n\right\}$. Then there
	exists  a positive constant $\delta=\delta(n,\Lambda_1, \Lambda_2, \pd,q(\cdot), w)$ such that the following holds. If the domain $\Omega$ is a $(\delta, R_0)$ Reifenberg flat domain with $ R_0>0$, and the nonlinearity $a$  satisfies \eqref{eq1-functiona}, \eqref{eq1s-functiona} and the small BMO condition \eqref{eq2-Assumption}, then for any weak solution $u\in W^{1,\pd}(\Om)$ to the problem \eqref{Maineq}, the following estimate holds true
	\begin{equation}\label{eq-mainthm}
	\int_{\Om} |Du|^{q(x)}w(x)dx \leq C\left[(|\mu|(\Om)^{\f{\sigma_0}{\gamma_1-1}}+|\Om|)^{n+1} + \int_{\Om} |\Mu|^{\f{q(x)}{p(x)-1}}w(x)dx\right]
	\end{equation}
	where $C$ is a constant depending on $n,\Lambda_1, \Lambda_2, \pd, q(\cdot), R_0,w, \sigma_0$.
	
	Or equivalently, we have
	\begin{equation}\label{eqs-mainthm}
	\int_{\Om} |Du|^{(p(x)-1)q(x)}w(x)dx \lesi \left[(|\mu|(\Om)^{\f{\sigma_0}{\gamma_1-1}}+|\Om|)^{n+1} + \int_{\Om} |\Mu|^{q(x)}w(x)dx\right].
	\end{equation}
	
\end{thm}
\begin{rem}
	In the particular case when $q(x)\equiv q\in (0,\vc)$, the term $(|\mu|(\Om)^{\f{\sigma_0}{\gamma_1-1}}+|\mu|(\Om))^{n+1}$ in \eqref{eq-mainthm} can be removed. More precisely, in this case we have
	$$
	\int_{\Om} |Du|^{q}w(x)dx \lesi \int_{\Om} |\Mu|^{\f{q}{p(x)-1}}w(x)dx.
	$$
	The proof can be done in the same manner as that of Theorem \ref{mainthm1}. However, we do not pursue it and we would leave it to the interested reader.
\end{rem} 
We now have the following consequences of Theorem \ref{mainthm1}.

\begin{cor}\label{mainthm2}
	Let $q(\cdot)$ be defined as in \eqref{cond1-qx}, \eqref{cond2-qx} and  \eqref{cond3-qx} with $1<\gamma_3\leq \gamma_4<n$. Then there
	exists  a positive constant $\delta=\delta(n,\Lambda_1, \Lambda_2, \pd,q(\cdot))$ such that the following holds. If the domain $\Omega$ is a $(\delta, R_0)$ Reifenberg flat domain with $ R_0>0$, and the nonlinearity  $a$ satisfies \eqref{eq1-functiona}, \eqref{eq1s-functiona} and the small BMO condition \eqref{eq2-Assumption}, then for any weak solution $u\in W^{1,\pd}_0(\Om)$ to the problem \eqref{Maineq} we obtain that
	\begin{equation}\label{eq2-mainthm}
	d\mu=fdx, f\in L^{q(\cdot)}(\Om)\Rightarrow |Du|^{\pd-1}\in L^{\f{nq(x)}{n-q(x)}}(\Om).
	\end{equation}
	
\end{cor}

In particular case when $q(x)$ is independent of $x$, Theorem \ref{mainthm1} deduces the following result.
\begin{cor}\label{mainthm3}
	Let $q\in (0,\vc)$, $w\in A_{\vc}$ and $0<\sigma_0<\min\left\{\f{n(\gamma_1-1)}{n-1},n\right\}$. Then there
	exists  a positive constant $\delta=\delta(n,\Lambda_1, \Lambda_2, \pd,q,w)$ such that the following holds. If the domain $\Omega$ is a $(\delta, R_0)$ Reifenberg flat domain with $ R_0>0$, and the nonlinearity $a$ satisfies \eqref{eq1-functiona}, \eqref{eq1s-functiona} and the small BMO condition \eqref{eq2-Assumption}, then for any weak solution $u\in W^{1,\pd}_0(\Om)$ to the problem \eqref{Maineq} the following estimate holds true
	\begin{equation}\label{eq3-mainthm}
	\Big\||Du|^{p(\cdot)-1}|\Big\|_{L^q_w(\Om)} \leq C\left[(|\mu|(\Om)^{\f{\sigma_0}{\gamma_1-1}}+|\Om|)^{\f{n+1}{q}} + \Big\|\Mu\Big\|_{L^q_w(\Om)}\right]
	\end{equation}
	where $C$ is a constant depending on $n,\Lambda_1, \Lambda_2, \pd, q, R_0,w, \sigma_0$.
	
	As a consequence, for $\f{1}{r}-\f{1}{q}=\f{1}{n}$ and $w^q\in A_{1+q/r'}$  if $d\mu=fdx, f \in L^r_{w^r}(\Om)$, then  we have
	\begin{equation}\label{eq3s-mainthm}
	\Big\||Du|^{p(\cdot)-1}|\Big\|_{L^q_{w^q}(\Om)} \leq C\left[(\|f\|_{L^1(\Om)}^{\f{\sigma_0}{\gamma_1-1}}+|\Om|)^{\f{n+1}{q}} + \Big\|f\Big\|_{L^r_{w^r}(\Om)}\right].
	\end{equation}
\end{cor}

We note that the estimate \eqref{eq3-mainthm} not only gives the $L^q$-weighted estimate for $|Du|^{p(x)-1}$ but also implies the  
estimate on Morrey space for $|Du|^{p(x)-1}$. We now recall the definition of Morrey space. 

For $0<q<\vc$ and $0<\lambda<n$, the  Morrey function spaces $L^{q;\lambda}(\Omega)$ is defined as the set of all measurable functions $f$ such that
$$
\|f\|_{L^{q; \lambda}(\Omega)}=\sup_{x\in \Omega}\sup_{0<r\leq {\rm diam}\Omega}\f{1}{r^{\lambda/q}}\|f\|_{L^{q}(B_{r}(x)\cap \Omega)}<\vc.
$$
Using a standard argument, see for example \cite{Ph}, from the weighted estimate \eqref{eq3-mainthm} we obtain the following Morrey space estimate.
\begin{cor}\label{mainthm4}
	Let $q\in (0,\vc)$, $\lambda\in (0,n)$ and $0<\sigma_0<\min\left\{\f{n(\gamma_1-1)}{n-1},n\right\}$. Then there
	exists  a positive constant $\delta=\delta(n,\Lambda_1, \Lambda_2, \pd,q,\lambda)$ such that the following holds. If the domain $\Omega$ is a $(\delta, R_0)$ Reifenberg flat domain with $ R_0>0$, and the nonlinearity $a$  satisfies \eqref{eq1-functiona}, \eqref{eq1s-functiona} and the small BMO condition \eqref{eq2-Assumption}, then for any weak solution $u\in W_0^{1,\pd}(\Om)$ to the problem \eqref{Maineq}, the following estimate holds true
	\begin{equation}\label{eq4-mainthm}
	\Big\||Du|^{p(\cdot)-1}|\Big\|_{L^{q;\lambda}(\Om)} \leq C\left[(|\mu|(\Om)^{\f{\sigma_0}{\gamma_1-1}}+|\Om|)^{\f{n+1}{q}} + \Big\|\Mu\Big\|_{L^{q;\lambda}(\Om)}\right]
	\end{equation}
	where $C$ is a constant depending on $n,\Lambda_1, \Lambda_2, \pd, q, R_0,\sigma_0$.
	
\end{cor}

In general, if the measure $\mu$ is merely a Radon measure with finite total mass, the weak solution $u\in W^{1,\pd}_0(\Om)$ to \eqref{Maineq} may not exist. In this situation, we employ the notion of SOLAs (Solution Obtained as Limit of Approximations). It is well-known that these solution may not be in $W^{1,\pd}_0(\Om)$, but in $W^{1,\pd -1}_0(\Om)$. In the particular case if $\mu\in W^{-1,\pd}(\Om)$, the dual space of $W^{1,\pd}_0(\Om)$, it is well-known that there exists a unique solution $u\in W^{1,\pd}_0(\Om)$ to \eqref{Maineq}, and in this case the SOLA and the weak solution to \eqref{Maineq} coincide.  See for example \cite{BH}.  From the above results, by a standard approximation procedure as in \cite{BH} we are able to obtain
\begin{thm}
	\label{mainthm5}
	Let $u\in W^{1,\pd -1}_0(\Om)$ be a SOLA to \eqref{Maineq}. Assume that all assumptions in the respective statements hold true. Then Theorem \ref{mainthm1} and Corollaries \ref{mainthm2}, \ref{mainthm3} and \ref{mainthm4} hold true. 
\end{thm}




\section{Approximation results}
In this section, we always assume that the nonlinearity satisfies \eqref{eq1-functiona}, \eqref{eq1s-functiona}, the small BMO norm condition \eqref{eq2-Assumption}  and the domain $\Om$ is a $(\delta, R_0)$ Reifenberg flat domain.

Let $u$ be a weak solution to the problem \eqref{Maineq}. We now fix $0<\sigma_0<\min\left\{\f{n(\gamma_1-1)}{n-1},n\right\}$. Then by a standard argument as in the proof of \cite[Theorem 1]{BG} there exists $C=K(n,\sigma_0,\Lambda_1, \gamma_2, \Om)>1$ so that 
\begin{equation*}
\int_{\Om}|Du|^{\sigma_0} dx  \leq C(|\mu|(\Om)^{\f{\sigma_0}{\gamma_1-1}}+|\Om|).
\end{equation*}
Hence, for any $0<q\leq \sigma_0$ we have
\begin{equation}\label{bounds-Du}
\int_{\Om}|Du|^{q} dx\leq \int_{\Om}|Du|^{\sigma_0}+1 dx  \leq C(|\mu|(\Om)^{\f{\sigma_0}{\gamma_1-1}}+|\Om|)=:K_0.
\end{equation}

For each $r>0$ and $x\in \Om$, we denote $\Om_r:=\Om_r(x)$ and set
$$
F(\mu,u,\Om_r)=\Big[\f{|\mu|(\Om_{r})}{r^{n-1}}\Big]^{\f{1}{p_{\Om_r}^+-1}} +\Big[\f{|\mu|(\Om_{r})}{r^{n-1}}\Big]\Big(\fint_{\Om_r}(|Du|+1)dx\Big)^{2-p_{\Om_r}^+}\chi_{\{p_{\Om_r}^+\leq 2\}}+1,
$$
where $p_{\Om_r}^+=\sup_{y\in \Om_r}p(y)$ and
$$
\chi_{\{p_{\Om_r}^+\leq 2\}}=\begin{cases}
1, \ \ \ &p_{\Om_r}^+\leq 2,\\
0, \ \ \ &p_{\Om_r}^+> 2.
\end{cases}
$$
\subsection{Interior Estimates}\label{InEs}
Let $x_0\in \Om$ and $0<R<\f{R_0\wedge R_\om\wedge K_0^{-1}}{10}$ so that $B_{2R}\equiv B_{2R}(x_0)\subset \Om$, where $R_\omega$ is a constant in \eqref{eq-R1}, and $a\wedge b=\min\{a,b\}$. We set 
$$
p_1=\inf_{x\in B_{2R}}p(x), \ \ \ p_2=\sup_{x\in B_{2R}}p(x).
$$

Let $u\in C^1(\Om)$ be a solution to \eqref{Maineq}. We now consider the following equation
\begin{equation}\label{eq1-interior}
\left\{
\begin{aligned}
&\di a(Dw, x)=0 \quad &\text{in} \quad B_{2R},\\
&w=u  \quad &\text{on} \quad \partial B_{2R},
\end{aligned}
\right.
\end{equation}
We have the following estimate.

\begin{prop}
	\label{prop1-interior}
	Let $w$ be a weak solution to \eqref{eq1-interior}. Then there exists a constant $C$  so that
	\begin{equation}
	\label{eq-Du-w}
	\fint_{B_{2R}}|D(u-w)|dx \leq CF(\mu,u,B_{2R}).
	\end{equation}
	As a consequence, we have
	\begin{equation}
	\label{eq1-Dw}
	\fint_{B_{2R}}|Dw| dx\leq C\left[ \fint_{B_{2R}}|Du| dx+F(\mu,u,B_{2R}) \right].
	\end{equation}
\end{prop}
\begin{proof}
	We consider two cases: $p_1\geq 2$ and $2-\f{1}{n}<p_1<2$.
	
	\noindent{\bf Case 1: $p_1\geq 2$.} In this case, the inequality \eqref{eq-Du-w} was proved in \cite[pp. 651--652]{BH}.\\

	\noindent{\bf Case 2: $p_1< 2$.} It was proved in \cite[Lemma 5.1]{BaH} that
	 $$
	\fint_{B_{2R}}|D(u-w)|dx\lesi \Big[\f{|\mu|(B_{2R})}{R^{n-1}}\Big]^{\f{1}{p_0-1}} +\Big[\f{|\mu|(B_{2R})}{R^{n-1}}\Big]\Big(\fint_{B_{2R}}(|Du|+1)dx\Big)^{2-p_0}+1,
	 $$ 
	where $p_0=p(x_0)$.
	
	Note that $|\mu|(B_{2R})\leq |\mu|(\Om)\leq K_0\leq R^{-1}$. Hence,
	$$
	\begin{aligned}
	\Big[\f{|\mu|(B_{2R})}{R^{n-1}}\Big]^{\f{1}{p_0-1}}&=\Big[\f{|\mu|(B_{2R})}{R^{n-1}}\Big]^{\f{1}{p_2-1}}\Big[\f{|\mu|(B_{2R})}{R^{n-1}}\Big]^{\f{p_2-p_0}{(p_0-1)(p_2-1)}}\\
	&\leq \Big[\f{|\mu|(B_{2R})}{R^{n-1}}\Big]^{\f{1}{p_2-1}}R^{-\f{n(p_2-p_0)}{(p_0-1)(p_2-1)}}\\
	&\leq \Big[\f{|\mu|(B_{2R})}{R^{n-1}}\Big]^{\f{1}{p_2-1}}R^{-\f{n\om(2R)}{(\gamma_1-1)^2}}\\
	&\leq C\Big[\f{|\mu|(B_{2R})}{R^{n-1}}\Big]^{\f{1}{p_2-1}},
	\end{aligned}
	$$
	noting that we used \eqref{eq-R1}  in the last inequality.
	
	On the other hand, by \eqref{bounds-Du} and \eqref{eq-R1},
	$$
	\begin{aligned}
	\Big[\f{|\mu|(B_{2R})}{R^{n-1}}\Big]\Big(\fint_{B_{2R}}(|Du|+1)dx\Big)^{2-p_0}&=\Big[\f{|\mu|(B_{2R})}{R^{n-1}}\Big]\Big(\fint_{B_{2R}}(|Du|+1)dx\Big)^{2-p_2}\Big(\fint_{B_{2R}}(|Du|+1)dx\Big)^{p_2-p_0}\\
	&\leq \Big[\f{|\mu|(B_{2R})}{R^{n-1}}\Big]\Big(\fint_{B_{2R}}(|Du|+1)dx\Big)^{2-p_2}(R^{-n}K_0)^{p_2-p_0}\\
	&\leq \Big[\f{|\mu|(B_{2R})}{R^{n-1}}\Big]\Big(\fint_{B_{2R}}(|Du|+1)dx\Big)^{2-p_2}R^{-(n+1)(p_2-p_0)}\\
	&\leq \Big[\f{|\mu|(B_{2R})}{R^{n-1}}\Big]\Big(\fint_{B_{2R}}(|Du|+1)dx\Big)^{2-p_2}R^{-(n+1)\om(2R)}\\
	&\leq \Big[\f{|\mu|(B_{2R})}{R^{n-1}}\Big]\Big(\fint_{B_{2R}}(|Du|+1)dx\Big)^{2-p_2}.
	\end{aligned}
	$$
	This completes our proof.
\end{proof}

	
We now record the higher integrability result in \cite[Theorem 5]{AM2}.
\begin{lem}
	\label{lem1-higherIntegrability} Let $w\in W^{1,p(\cdot)}(B_{2R})$ be a weak solution to \eqref{eq1-interior}. Then there exists a constant $\sigma^*=\sigma^*(n,\Lambda_1, \Lambda_2, \gamma_2, |\mu|(\Om), \Om)$ so that for $0\leq \sigma<\sigma^*$ and any $q\in (0,1]$ there exists $C>0$ such that
\begin{equation}\label{eq-higherinte}
	\Big(\fint_{B_R}|Dw|^{p(x)(1+\sigma)}dx\Big)^{1+\sigma}\leq C\left[\Big(\fint_{B_{2R}}|Dw|^{qp(x)}dx\Big)^{1/q}+1\right].
	\end{equation}
\end{lem}
\begin{proof}
It was proved in \cite[Theorem 5]{AM2} that for any $\sigma\in (0,\sigma_1)$ we have
$$
\Big(\fint_{B_R}|Dw|^{p(x)(1+\sigma)}dx\Big)^{1+\sigma}\leq C\left[\fint_{B_{2R}}|Dw|^{p(x)}dx +1\right],
$$
with 
$$
\sigma_1=\min\left\{1, c_1\Big(\int_{B_{2R}}|Dw|^{p(x)}dx +1\Big)^{-\f{4(p_2-p_1)}{p_1}}\right\},
$$	
for some constant $c_1>0$.

Note that we can get rid of the dependence of the constant $\sigma_1$ on $\int_{B_{2R}}|Dw|^{p(x)}dx$. To do this we recall an estimate in \cite[p. 654]{BH}
$$
\int_{B_R}|Dw|^{p(x)}dx\leq CR^{-\f{\alpha}{1-\beta}} \Big(\int_{B_{2R}}|Dw|dx +1\Big)^{\f{\gamma}{1-\beta}}, 
$$
where 
$$
\alpha=n\Big(\f{\kappa p_2}{p_1}-1\Big), \beta=\f{p_2}{p_1}\times \f{1-\f{\kappa}{p_1}}{1-\f{1}{p_1}}, \gamma=\f{p_2}{p_1}\times \f{\kappa -1}{1-\f{1}{p_1}}, \kappa=\sqrt \f{n+1}{n}.
$$
Hence,
$$
\Big(\int_{B_{2R}}|Dw|^{p(x)}dx +1\Big)^{\f{4(p_2-p_1)}{p_1}}\leq CR^{-\f{\alpha}{1-\beta}\times \f{4(p_2-p_1)}{p_1}} \Big(\int_{B_{2R}}|Dw|dx +1\Big)^{\f{\gamma}{1-\beta}\times \f{4(p_2-p_1)}{p_1}}.
$$
Moreover, observe that  
$$
\alpha\leq n\Big(\f{\kappa \gamma_2}{\gamma_1}-1\Big), \beta\leq \f{\gamma_2}{\gamma_1}\times \f{1-\f{\kappa}{\gamma_2}}{1-\f{1}{\gamma_2}}, \gamma\leq \f{\gamma_2}{\gamma_1}\times \f{\kappa -1}{1-\f{1}{\gamma_2}}.
$$
As a consequence,
$$
\Big(\int_{B_{2R}}|Dw|^{p(x)}dx +1\Big)^{\f{4(p_2-p_1)}{p_1}}\leq CR^{-c_1(p_2-p_1)} \Big(\int_{B_{2R}}|Dw|dx +1\Big)^{c_2},
$$
where $c_1, c_2$ are two constants independent of $Dw$ and $R$.

Note that from \eqref{cond2-px} we have
$$
R^{-c_1(p_2-p_1)}\leq CR^{-c_1\gamma(4R)}\leq C.
$$
On the other hand, from \eqref{bounds-Du} and \eqref{eq1-Dw} by a simple manipulation we get
$$
\int_{B_{2R}}|Dw|dx +1\leq c(\Om, |\mu|(\Om)).
$$
Therefore, there exists $\sigma^*=\sigma^*(n,\Lambda_1, \Lambda_2, \gamma_2, |\mu|(\Om), \Om)$ so that 
$$
\sigma_1>\sigma^*.
$$ 
Hence, the desired estimate follows from  Gehrings lemma in \cite[Theorem 6.7]{Giu}.
\end{proof}

Consider the nonlinearity $b(\cdot,\cdot)$ associated to $a(\cdot,\cdot)$ defined by
\begin{equation}\label{eq-b}
b(\xi,x)=\begin{cases}
(s^2+|\xi|^2)^{\f{p_2-p(x)}{2}}a(\xi,x), x\in B_{2R}\equiv B_{2R}(x_0),\\
(s^2+|\xi|^2)^{\f{p_2-p(x_0)}{2}}a(\xi,x_0), x\in  \Rn\backslash B_{2R}(x_0).
\end{cases}
\end{equation}
\begin{lem}\label{lem-nonlinearity b}
There exists $R_a>0$ so that for any $0<R<R_a$, the nonlinearity $b$ defined as above satisfies the following conditions:
\begin{equation}\label{eq1-functionb}
(s^2+|\xi|^2)^{1/2}|D_\xi b(\xi,x)|+|b(\xi,x)|\leq 3\Lambda_1 (s^2+|\xi|^2)^{\f{p_2-1}{2}},
\end{equation}
and
\begin{equation}\label{eq2-functionb}
\langle D_\xi b(\xi,x)\eta, \eta \rangle  \geq \f{\Lambda_2}{2} (s^2+|\xi|^2)^{\f{p_2-2}{2}}|\xi|^2,
\end{equation}
for all $x,\xi,\eta \in \mathbb{R}^n$.
\end{lem}
\begin{proof}
	The proof of this lemma is similar to that  in \cite[p. 13]{BOR}. For the convenience of reader, we sketch it here.
	
	From \eqref{eq1-functiona}, by a simple calculation we have
	$$
	\begin{aligned}
	(s^2+|\xi|^2)^{1/2}|D_\xi b(\xi,x)|+|b(\xi,x)|&\leq \Lambda_1(1+\gamma(4r))(s^2+|\xi|^2)^{\f{p_2-1}{2}}\\
	&:=\Lambda_3(s^2+|\xi|^2)^{\f{p_2-1}{2}}.
	\end{aligned}
	$$
	
	It suffices to verify \eqref{eq2-functionb}. We need to prove that there exist $R_a>0$ so that for any $0<R<R_a$, the estimate \eqref{eq2-functionb} holds true for all $x,\xi,\eta \in \mathbb{R}^n$.
	
	Indeed, if $x\in B_{2R}$, we have 
	$$
	\begin{aligned}
	 \langle D_\xi b(\xi,x)\eta,\eta\rangle =&(s^2+|\xi|^2)^{\f{p_2-p(x)}{2}} \langle D_\xi a(\xi,x)\eta,\eta\rangle +
	(p_2-p(x))|\xi|(s^2+|\xi|^2)^{\f{p_2-p(x)}{2}-1}  \langle D_\xi a(\xi,x)\eta,\eta\rangle \\
	:=& I_1+I_2.
	\end{aligned}
	$$
	By \eqref{eq1s-functiona}, one gets that
	$$
	I_1\geq \Lambda_2 (s^2+|\xi|^2)^{\f{p_2-2}{2}}|\eta|^2. 
	$$
	Applying \eqref{eq1-functiona}, we have
	$$
	\begin{aligned}
	I_2&\geq -\Lambda_1(p_2-p_1)(s^2+|\xi|^2)^{\f{p_2-p(x)}{2}-1}(s^2+|\xi|^2)^{\f{\px-1}{2}}|\eta|^2\\
	&\geq \omega(4R)\Lambda_1 (s^2+|\xi|^2)^{\f{p_2-2}{2}}|\eta|^2.
	\end{aligned}
		$$ 
		
	Hence,
	$$
	\langle D_\xi b(\xi,x)\eta,\eta\rangle  \geq [\Lambda_2-\omega(4R)\Lambda_1] (s^2+|\xi|^2)^{\f{p_2-2}{2}}|\eta|^2.
	$$
	Then the constant $R_a$ can be chosen as a number satisfying $\gamma(4R_a)<\f{\Lambda_2}{2\Lambda_1}$.\\

	The case $x\in B_{2R}^c$ can be argued similarly. Hence, we complete the proof.
\end{proof}

We now consider the following equation
\begin{equation}\label{eq2-interior}
\left\{
\begin{aligned}
&\di b(Dh, x)=0 \quad &\text{in}& \quad B_{R},\\
&h=w  \quad &\text{on}& \quad \partial B_{R},
\end{aligned}
\right.
\end{equation}
where $w$ is a weak solution to the problem \eqref{eq1-interior}.

\begin{prop}
	\label{prop2-inter}
	For any $\epsilon>0$ there exists $R_\epsilon$ depending on $\epsilon$ only so that  if $h$ is a weak solution to \eqref{eq2-interior} with $0<R<\f{R_\epsilon\wedge R_0\wedge R_\om\wedge R_a \wedge K_0^{-1}}{10}$, then we have
	\begin{equation}
	\label{eq-prop2}
	\Big(\fint_{B_R}|D(h-w)|^{p_2}dx\Big)^{1/p_2}\leq \epsilon  \Big[F(\mu,u,B_{2R})+\fint_{B_{2R}}|Du|dx\Big]. 
	\end{equation}
\end{prop}
\begin{proof}
	We consider two cases: $p_2< 2$ and $p_2\geq 2$.
	
	\noindent{\bf Case 1: $p_2<2$.} We first write
	$$
	|D(h-w)|^{p_2}= (s^2+|Dh|^2+|Dw|^2)^{-\f{p_2(p_2-2)}{4}}(s^2+|Dh|^2+|Dw|^2)^{\f{p_2(p_2-2)}{4}}|D(h-w)|^{p_2}.
	$$
	For $\tau_1>0$, using Young's inequality we obtain
	\begin{equation}\label{eq1-proof prop1 inter}
	\begin{aligned}
	\fint_{B_R}&|D(h-w)|^{p_2} dx\\
	&\leq \tau_1\fint_{B_{R}}(s^2+|Dh|^2+|Dw|^2)^{\f{p_2}{2}} dx + c(\tau_1)\fint_{B_{R}}(s^2+|Dh|^2+|Dw|^2)^{\f{p_2-2}{2}}|D(h-w)|^2dx\\
	&\leq c\tau_1\Big[\fint_{B_{R}}|D(h-w)|^{p_2}dxdt +\fint_{B_{R}}|Dw|^{p_2}dx +1\Big]\\
	& \ \ + c(\tau_1)\fint_{B_{R}}(s^2+|Dh|^2+|Dw|^2)^{\f{\px-2}{2}}|D(h-w)|^2dx.
	\end{aligned}
	\end{equation}
	
	Note that, by \eqref{eq2-functionb}, we have
	\begin{equation}\label{eq2-proof prop1 inter}
	\begin{aligned}
	\fint_{B_{R}}(s^2+|Dh|^2+|Dw|^2)^{\f{p_2-2}{2}}|D(h-w)|^2dx \leq C\fint_{B_{R}} \langle b(Dh,x)-b(Dw,x), Dh-Dw \rangle dx.
	\end{aligned}
	\end{equation}
	Substituting \eqref{eq2-proof prop1 inter} into \eqref{eq1-proof prop1 inter}, we obtain
	\begin{equation}\label{eq2s-proof prop1 inter}
	\begin{aligned}
	\fint_{B_R}|D(h-w)|^{p_2} dx\leq&  c\tau_1\fint_{B_R}|D(h-w)|^{p_2} dx +c\tau_1\Big(\fint_{B_{R}}|Dw|^{p_2}dx+1\Big)\\
	& + c(\tau_1)\fint_{B_{R}} \langle b(Dh,x)-b(Dw,x), Dh-Dw \rangle dx.
	\end{aligned}
	\end{equation}
	Moreover, from the definition of $b(x,\xi)$ and \eqref{eq1-functiona} we have
	$$
	\begin{aligned}
	\fint_{B_{R}} \langle b(Dh,x)-b(Dw,x), D(h-w) \rangle dx
	&=\fint_{B_{r}} \langle a(Dw,x)-b(Dw,x), D(h-w) \rangle dx\\
	&=\fint_{B_{r}} \Big|(s^2+|Dw|^2)^{\f{p_2-\px}{2}}-1\Big|\langle a(Dw,x), D(h-w) \rangle dx\\
	&\leq C\fint_{B_{r}} \Big[1-(s^2+|Dw|^2)^{\f{p_2-\px}{2}}\Big] (s+|Dw|)^{\px -1} |D(h-w)| dx.
	\end{aligned}
	$$
	Using the Mean Value Theorem, we obtain
	$$
	\begin{aligned}
	\fint_{B_{R}} &\langle b(Dh,x)-b(Dw,x), D(h-w) \rangle dx\\
	&\leq C\fint_{B_{R}}\sup_{\theta\in [0,1]} (p_2-p(x))\log(s^2+|Dw|^2)(s^2+|Dw|^2)^{\f{\theta(p_2-\px)}{2}}(s^2+|Dw|)^{\px -1} |D(h-w)| dx\\
	&\leq C\om(4R)\fint_{B_{R}} \log(1+|Dw|)(1+|Dw|)^{p_2-\px}(1+|Dw|)^{\px -1} |D(h-w)| dx,
	\end{aligned}
	$$
	where in the last inequality we used the log-H\"older condition \eqref{cond1-px}.
	
	Hence,
	$$\fint_{B_{R}} \langle b(Dh,x)-b(Dw,x), D(h-w) \rangle dx\leq C\omega(4R)\fint_{B_{R}} \log(1+|Dw|)(1+|Dw|)^{p_2-1} |D(h-w)| dx.
	$$
    Using Young's inequality, for $\tau_2>0$, which will be fixed later, we obtain
	\begin{equation}
	\label{eq2-prop2-inter}
	\begin{aligned}
	\fint_{B_{R}} &\langle b(Dh,x)-b(Dw,x), D(h-w) \rangle dx\\
	&\leq \tau_2 \omega(4R)\fint_{B_{R}}|D(h-w)|^{p_2} dx + c(\tau_2)\omega(4R)\fint_{B_{R}} \log(1+|Dw|)^{\f{p_2}{p_2-1}}(1+|Dw|)^{p_2}dx.
	\end{aligned}
	\end{equation}
	
	We now apply the inequality $\log(1+t)^{\f{p_2}{p_2-1}}\leq c\alpha^{-\f{p_2}{p_2-1}}(1+t)^{\alpha/4}$ for $t>0$ and $\alpha\in (0,1)$ with $t=|Dw|$ and $\alpha= p_2\om(4R)/4$ to conclude that
	$$
	\log(1+|Dw|)^{\f{p_2}{p_2-1}}\leq c(p_2\om(4R))^{-\f{p_2}{p_2-1}}(1+|Dw|)^{p_2\om(4R)}.
	$$	
	
	
	Substistuting this into \eqref{eq2-prop2-inter} we get that 
	$$
	\om(4R)\fint_{B_{R}} \log(1+|Dw|)^{\f{p_2}{p_2-1}}(1+|Dw|)^{p_2}dx\leq c\om(4R)^{\f{1}{p_2-1}}\fint_{B_{R}}(1+|Dw|)^{p_2(1+\gamma(4R))}dx.
$$
Moreover, it is obvious that 
$$
p_2(1+\gamma(4R))\leq (p(x)+\mod{(4R)})(1+\om(4R))\leq p(x)(1+3\om(4R)).
$$
Hence,
$$
\om(4R)\fint_{B_{R}} \log(1+|Dw|)^{\f{p_2}{p_2-1}}(1+|Dw|)^{p_2}dx\leq c\om(4R)^{\f{1}{p_2-1}}\fint_{B_{R}}(1+|Dw|)^{p(x)(1+3\gamma(4R))}dx.
$$	
This along with Lemma \ref{lem1-higherIntegrability} gives
\begin{equation}\label{eq-estimate p2}
\begin{aligned}
\om(4R)\fint_{B_{R}} \log(1+|Dw|)^{\f{p_2}{p_2-1}}(1+|Dw|)^{p_2}dx&\leq c\om(4R)^{\f{1}{p_2-1}}\Big(\fint_{B_{R}}|Dw|^{\f{\px}{p_2}}dx+1\Big)^{p_2[1+3\gamma(4R)]}\\	
	&\leq c\om(4R)^{\f{1}{p_2-1}}\Big(\fint_{B_{R}}|Dw|dx+1\Big)^{p_2(1+3\gamma(4R))},
	\end{aligned}
	\end{equation}
	as long as $3\om(4R)<\sigma^*$.

	We note that from \eqref{eq-R1}, 
	$$
	|B_R|^{-3p_2\omega(4R)}\lesi |R|^{-3\omega(4R) \gamma_2 n}\lesi 1. 
	$$
	As a consequence,
	\begin{equation}\label{eq2s-estimate p2}
	\begin{aligned}
	\omega(4R)\fint_{B_{R}} \log(1+|Dw|)^{\f{p_2}{p_2-1}}(1+|Dw|)^{p_2}dx&\leq c\om(4R)^{\f{1}{p_2-1}}\Big(\fint_{B_{R}}|Dw|dx+1\Big)^{p_2}\Big(\int_{B_{R}}|Dw|dx+R^n\Big)^{3\om(4R)}.
	\end{aligned}
	\end{equation}
	Moreover, from \eqref{eq1-Dw}, \eqref{bounds-Du} and the fact that $|\mu|(\Om)< K_0$ we have
	$$
	\begin{aligned}
	\fint_{B_{R}}|Dw|dx+R^n&\lesi \int_{B_{2R}}|Du| dx+R^n\Big[\f{|\mu|(B_{2R})}{R^{n-1}}\Big]^{\f{1}{p_2-1}} +R^n\Big[\f{|\mu|(B_{2R})}{R^{n-1}}\Big]\Big(\fint_{B_{2R}}(|Du|+1)dx\Big)^{2-p_2}+R^n\\
	&\lesi K_0 + K_0^{\f{1}{p_2-1}}+K_0 R^{-n(2-p_2)}K_0+1\\
	&\lesi R^{-1} + R^{-\f{1}{p_2-1}}+R^{-n(2-p_2)-2}+1\\
	&\lesi R^{-(n+2)},
	\end{aligned}
	$$
	where in the third inequality we used the fact that $K_0\leq R^{-1}$.
	
	This together with \eqref{eq-R1} gives
	$$
	\Big(\int_{B_{R}}|Dw|dx+R^n\Big)^{3\om(4R)}\lesi R^{-3\om(4R)(n+2)}\lesi 1.
	$$
	Inserting this into \eqref{eq2s-estimate p2} we obtain
	\begin{equation}\label{eq2-estimate p2}
	\begin{aligned}
	\omega(4R)\fint_{B_{R}} \log(1+|Dw|)^{\f{p_2}{p_2-1}}(1+|Dw|)^{p_2}dx
	&\leq c\om(4R)^{\f{1}{p_2-1}}\Big(\fint_{B_{R}}|Dw|dx+1\Big)^{p_2}.
	\end{aligned}
	\end{equation}
	
	We now combine \eqref{eq2-estimate p2} and  \eqref{eq2-prop2-inter} to imply that
	$$
	\begin{aligned}
	\fint_{B_{r}} &\langle b(Dh,x)-b(Dw,x), Dh-Dw \rangle dx\\
	&\leq \tau_2 \om(4R)\fint_{B_{r}}|D(h-w)|^{p_2} dx + c(\tau_1)\om(4R)^{\f{1}{p_2-1}}\Big(\fint_{B_{r}}|Dw|dx+1\Big)^{p_2}.
	\end{aligned}
	$$
	This in combination with \eqref{eq2s-proof prop1 inter} yields
	\begin{equation}\label{eq4-prop2-inter}
	\begin{aligned}
	\fint_{B_{R}}|D(h-w)|^{p_2} dx
	\leq& \tau_1\fint_{B_{R}}|D(h-w)|^{p_2} dx +\tau_1\Big(\fint_{B_{R}}|Dw|^{p_2} dx+1\Big)\\
	&+ \tau_2 \om(4R)\fint_{B_{r}}|D(h-w)|^{p_2} dx+ c(\tau_2)\om(4R)^{\f{1}{p_2-1}}\Big(\fint_{B_{r}}|Dw|dx+1\Big)^{p_2}.
	\end{aligned}
	\end{equation}
	On the other hand, arguing similarly to \eqref{eq2-estimate p2}, we arrive at
	$$
	\Big(\fint_{B_{R}}|Dw|^{p_2} dx+1\Big)\leq C\Big(\fint_{B_{R}}|Dw| dx+1\Big)^{p_2}.
	$$
	Inserting this into \eqref{eq4-prop2-inter}, we conclude that
	$$
	\begin{aligned}
		\fint_{B_{R}}|D(h-w)|^{p_2} dx
		\leq& \tau_1\fint_{B_{R}}|D(h-w)|^{p_2} dx + \tau_2 \om(4R)\fint_{B_{r}}|D(h-w)|^{p_2} dx\\
		&+ (c(\tau_2)\om(4R)^{\f{1}{p_2-1}}+\tau_1)\Big(\fint_{B_{r}}|Dw|dx+1\Big)^{p_2}.
	\end{aligned}
	$$
	Hence, the desired estimate \eqref{eq-prop2} follows from the inequality above by choosing $\tau_1, \tau_2$ and $R_\varepsilon$ to be sufficiently small.
	
	\bigskip
	
	\noindent{\bf Case 2: $p_2\geq 2$.} Observe that 
	$$
	\fint_{B_{R}}|D(h-w)|^{p_2}dx\lesi \fint_{B_{R}} \langle b(Dh,x)-b(Dw,x), D(h-w) \rangle dx.
	$$
	At this stage, repeating the argument used in Case 1, the desired estimate \eqref{eq-prop2} is proved.
\end{proof}

We now consider the following equation
\begin{equation}\label{eq3-interior}
	\left\{
	\begin{aligned}
		&\di \bar{b}_{B_R}(Dv)=0 \quad &\text{in}& \quad B_{R},\\
		&v=h  \quad &\text{on}& \quad \partial B_{R},
	\end{aligned}
	\right.
\end{equation}
where $h$ is a weak solution to the problem \eqref{eq2-interior}.

We have the following estimate.
\begin{prop}
	\label{prop3-interior}
	For any $\epsilon > 0$ there exist $\delta>0$ and $R_\epsilon>0$ so that if $v$ be a weak solution to the problem \eqref{eq3-interior} with $0<R<\f{R_\epsilon\wedge R_0\wedge R_\om\wedge R_a \wedge K_0^{-1}}{10}$, then we have
	\begin{equation}\label{eq1-prop3-inter}
	\|Dv\|_{L^\vc(B_{R/2})}\leq C\Big[F(\mu,u,B_{2R})+\fint_{B_{2R}}|Du|dx\Big],
	\end{equation}
	and
	\begin{equation}\label{eq2-prop3-inter}
	\Big(\fint_{B_{R/2}}|D(v-h)|^{p_2}dx\Big)^{1/p_2}\leq \epsilon\Big[F(\mu,u,B_{2R})+\fint_{B_{2R}}|Du|dx\Big].
	\end{equation}
\end{prop}
\begin{proof}
	Let $R_\epsilon$ as in Proposition \eqref{prop2-inter}. We take care of \eqref{eq2-prop3-inter} first. Taking $v-h$ as a test function, it can be verified that
	$$
	\langle b(Dv,x)-b(Dh,x), D(h-v)\rangle=\langle b(Dv,x)-\bar{b}_{B_R}(Dv), D(h-v)\rangle.
	$$
	We take care of \eqref{eq2-prop3-inter} first. We give the proof as $p_2<2$, since the case $p_2\geq 2$ can be done in the same manner and even easier. Taking $v-h$ as a test function, it can be verified that
	$$
	\langle b(Dv,x)-b(Dh,x), D(h-v)\rangle=\langle b(Dv,x)-\bar{b}_{B_R}(Dv), D(h-v)\rangle.
	$$
	For $p_2<2$, arguing similarly to \eqref{eq1-proof prop1 inter}-\eqref{eq2-proof prop1 inter}, we find that for   $\tau_1>0$, we have
	$$
	\begin{aligned}
	\fint_{B_{R}}|D(v-h)|^{p_2}&\leq \tau_1\left[\fint_{B_R}|h|^{p_2} dx+1\right]+c(\tau_1)\fint_{B_{R}} \langle b(Dv,x)-b(Dh,x), D(h-v)\rangle dx\\
	&=\tau_1\left[\fint_{B_R}|h|^{p_2} dx+1\right]+c(\tau_1)\fint_{B_{R}}\langle b(Dv,x)-\bar{b}_{B_R}(Dv), D(h-v)\rangle dx\\
	&\leq \tau_1\left[\fint_{B_R}|h|^{p_2} dx+1\right]+c(\tau_1)\fint_{B_{R}}\Theta(a,B_R)(x)(\mu+|Dv|)^{p_2-1}|D(h-v)|dx.
	\end{aligned}
	$$
	
	Using Young's inequality we obtain, for $\tau_2>0$,
	$$
	\begin{aligned}
	\fint_{B_{R}}\Theta(a,B_R)(x)&(\mu+|Dv|)^{p_2-1}|D(h-v)|dx\\
	&\leq \tau_2\fint_{B_{R/2}}|D(v-h)|^{p_2}dx + c(\tau_2)\fint_{B_{R/2}}\Theta(a,B_R)^{\f{p_2}{p_2-1}}(1+|Dv|)^{p_2}dx.
	\end{aligned}
	$$
	By the standard higher integrability result for the problem \eqref{eq3-interior}, there exists a constant $\sigma_2>0$ so that
	$$
	\Big(\fint_{B_{R/2}}(1+|Dv|)^{p_2(1+\sigma_2)}dx\Big)^{\f{1}{1+\sigma_2}}\lesi \fint_{B_{R}}(1+|Dv|)^{p_2}dx.
	$$
	This along with Remark \ref{rem1} yields
	$$
	\begin{aligned}
	\fint_{B_{R/2}}\Theta(a,B_R)^{\f{p_2}{p_2-1}}(1+|Dv|)^{p_2}&\leq \Big(\fint_{B_{R/2}}\Theta(a,B_R)^{\f{p_2}{p_2-1}\f{1+\sigma_2}{\sigma_2}}\Big)^{\f{\sigma_2}{1+\sigma_2}}\Big(\fint_{B_{R/2}}(1+|Dv|)^{p_2(1+\sigma_2)}\Big)^{\f{1}{1+\sigma_2}}\\
	&\leq \mathcal{O}(\delta)\Big(\fint_{B_R}|Dv|^{p_2}dx +1\Big)\\
	&\leq \mathcal{O}(\delta)\Big(\fint_{B_R}|Dh|^{p_2}dx +1\Big),
	\end{aligned}
	$$
	where in the last inequality we used the standard $L^{p_2}$-boundedness of \eqref{eq3-interior}.
	
	Putting these three estimates in hand, we conclude that
	\begin{equation}\label{eq4-prop3-inter}
	\begin{aligned}
	\fint_{B_{R/2}}|D(v-h)|^{p_2}
	&\lesi \tau_2 \fint_{B_{R/2}}|D(v-h)|^{p_2}dx+(\mathcal{O}(\delta)+\tau_1)\Big(\int_{B_R}|Dh|^{p_2}dx +1\Big).
	\end{aligned}
	\end{equation}
	We now write
	$$
	\fint_{B_R}|Dh|^{p_2}dx\leq C\fint_{B_R}|D(h-w)|^{p_2}dx+C\fint_{B_R}|Dw|^{p_2}dx.	
	$$
	Using \eqref{eq-prop2}, we can dominate
	$$
	\fint_{B_R}|D(h-w)|^{p_2}dx\leq \epsilon  \Big[F(\mu,u,B_{2R})+\fint_{B_{2R}}|Du|dx\Big]^{p_2}, 
	$$
	as long as $R<R_{\epsilon}$.

	Arguing similarly to the proof of \eqref{eq2-estimate p2},
	$$
	\fint_{B_R}|Dw|^{p_2}dx\leq C \Big(\fint_{B_R}|Dw(x)|dx +1\Big)^{p_2}\leq C\Big[F(\mu,u,B_{2R})+\fint_{B_{2R}}|Du|dx\Big]^{p_2}. 
	$$
	Consequently,
	\begin{equation}\label{eq-Dhp2}
	\fint_{B_R}|Dh|^{p_2}dx +1\leq C  \Big(\fint_{B_R}|Dw(x)|dx +1\Big)^{p_2}\leq C \Big[F(\mu,u,B_{2R})+\fint_{B_{2R}}|Du|dx\Big]^{p_2}.
	\end{equation}
	Putting this into \eqref{eq4-prop3-inter}, we have
	$$
	\begin{aligned}
	\fint_{B_{R}}|D(v-h)|^{p_2}	&\lesi \tau_2 \fint_{B_{R}}|D(v-h)|^{p_2}dx+(\mathcal{O}(\delta)+\tau_1)\Big[F(\mu,u,B_{2R})+\fint_{B_{2R}}|Du|dx\Big]^{p_2}.	
	\end{aligned}
	$$
	Hence, \eqref{eq2-prop3-inter} follows by choosing $\tau$ and $\delta$ to be sufficiently small.
	\medskip
	
	We now turn to estimate \eqref{eq1-prop3-inter}.
	From the well-known H\"older estimate, see for example \cite{L1, L2}, we have
	$$
	\|Dv\|_{L^\vc(B_{R/2})}^{p_2}\lesi\fint_{B_R}|Dv|^{p_2}dx +1 \lesi \fint_{B_R}|D(h-v)|^{p_2}dx+ \fint_{B_R}|Dh|^{p_2}dx+1. 
	$$
	Using \eqref{eq2-prop3-inter} and \eqref{eq-Dhp2}, we imply \eqref{eq1-prop3-inter}.
\end{proof}

\subsection{Boundary estimates}
We now consider the case $x_0\in \partial\Om$, and $0<R<\f{R_0\wedge R_w\wedge K_0^{-1}}{10}$. We set 
$$
p_1=\inf_{x\in \Om_{2R}(x_0)}p(x), \ \ \ p_2=\sup_{x\in \Om_{2R}(x_0)}p(x).
$$

Let $u\in W^{1,p(\cdot)}_0(\Om)$ be a weak solution to \eqref{Maineq}. We now consider the following equation
\begin{equation}\label{eq1-boundary}
\left\{
\begin{aligned}
&\di a(Dw, x)=0 \quad &\text{in} \quad \Om_{2R}(x_0),\\
&w=u  \quad &\text{on} \quad \partial \,\Om_{2R}(x_0),
\end{aligned}
\right.
\end{equation}
We have the following estimate.

\begin{prop}
	\label{prop1-boundary}
	Let $w$ be a weak solution to \eqref{eq1-boundary}. Then there exists $C$  so that
	\begin{equation}\label{eq-Du-w-boundary}
	\fint_{\Om_{2R}(x_0)}|D(u-w)|dx\leq CF(\mu,u,\Om_{2R}(x_0)).
	\end{equation}
	As a consequence,
	\begin{equation}\label{eq-Dw-boundary}
	\fint_{\Om_{2R}(x_0)}|Dw| dx\leq C\Big[F(\mu,u,\Om_{2R}(x_0))+\fint_{\Om_{2R}(x_0)}|Du|dx\Big].
	\end{equation}
\end{prop}
\begin{proof}
	The proof of the proposition is similar to that of Proposition \ref{prop1-interior}. We omit details. 
\end{proof}

Like the higher integrability result in Lemma \ref{lem1-higherIntegrability}, the similar result still holds true near the boundary of the Reifenberg domains.

\begin{lem}\label{lem2-higherIntegrability}
	\label{lem2-higherIntegrability} Let $w\in W^{1,p(\cdot)}(\Om_{2r}(x_0))$ be a weak solution to the problem \eqref{eq1-boundary} with $r\leq \min\{R_\omega, R_0/2\}$. Then there exists a constant, which we still denote $\sigma^*=\sigma^*(n,\Lambda_1, \Lambda_2, \gamma_2, \mu, \Om)$, so that for $0\leq \sigma<\sigma^*$ and any $q\in (0,1]$ there exists $C>0$ such that
	\begin{equation}\label{eq-higherinte-boundary}
	\Big(\fint_{\Om_R(x_0)}|Dw|^{p(x)(1+\sigma)}dx\Big)^{1+\sigma}\leq C\left[\Big(\fint_{\Om_{2R}(x_0)}|Dw|^{qp(x)}dx\Big)^{1/q}+1\right].
	\end{equation}
\end{lem}
\begin{proof}
	The proof of this lemma is similar to that of Lemma \ref{lem1-higherIntegrability} and we omit details.
\end{proof}
We now consider the following equation
\begin{equation}\label{eq2-boundary}
\left\{
\begin{aligned}
&\di b(Dh, x)=0 \quad &\text{in}& \quad \Om_{R}(x_0),\\
&h=w  \quad &\text{on}& \quad \partial \Om_{R}(x_0),
\end{aligned}
\right.
\end{equation}
where $w$ is a weak solution to the problem \eqref{eq1-boundary}, and $b$ is a nonlinearity defined similarly to \eqref{eq-b} but $\Om_{R}(x_0)$ taking place of $B_{2R}(x_0)$.

Arguing similarly to Proposition \ref{prop2-inter}, we can prove that:
\begin{prop}
	\label{prop2-boundary}
	For any $\epsilon>0$ there exists $R_\epsilon$ so that  if $h$ is a weak solution to \eqref{eq2-boundary} with $0<R<\f{R_{\epsilon}\wedge R_0\wedge R_\om\wedge R_a \wedge K_0^{-1}}{10}$, then we have
	\begin{equation}
	\label{eq-prop2-boundary}
	\fint_{\Om_R(x_0)}|D(h-w)|^{p_2}dx\leq \epsilon  F(\mu,u,\Om_{2R}(x_0))^{p_2}.
	\end{equation}
\end{prop}

We now assume that $0<\delta<1/50$. Since $x_0\in \partial\Om$, there exists a new coordinate system whose variables are still denoted by $(x_1,\ldots,x_n)$ such that in this coordinate system the origin is some interior point of $\Om$, $x_0=(0,\ldots,0,-\f{\delta R}{2(1-\delta)})$ and
\begin{equation}\label{eq1-new coordinate}
B_{R/2}^+ \subset B_{R/2}\cap \Om\subset B_{R/2}\cap \{x: x_n>-3\delta R\}.
\end{equation}
Note that due to $\delta\in (0,1/50)$, we further obtain 
\begin{equation}\label{eq2-new coordinate}
B_{3R/8}\subset B_{R/4}(x_0)\subset B_{R/2}\subset B_R(x_0).
\end{equation}

We now consider the following equations
\begin{equation}\label{eq3s-boundary}
\left\{
\begin{aligned}
&\di \bar{b}_{B_R}(DV)=0 \quad &\text{in}& \quad \Om_{R/2},\\
&V=h  \quad &\text{on}& \quad \partial\Om_{R/2},
\end{aligned}
\right.
\end{equation}
and
\begin{equation}\label{eq3-boundary}
\left\{
\begin{aligned}
&\di \bar{b}_{B_r}(Dv)=0 \quad &\text{in}& \quad B^+_{R/2},\\
&v=0  \quad &\text{on}& \quad T_{R/2}:=B_{R/2}\cap\{x_n=0\}.
\end{aligned}
\right.
\end{equation}

Similarly to Proposition \ref{prop3-interior} we have the following estimate.

\begin{prop}
	\label{prop3-boundary}
	For any $\epsilon > 0$ there exist $\delta>0$ and $R_\epsilon>0$ so that if $V$ be a weak solution to the problem \eqref{eq3s-boundary} with $0<R<R_\epsilon$, then we have
	\begin{equation}\label{eq2-prop3-boundary}
	\Big(\fint_{\Om_{R/4}(x_0)}|D(V-h)|^{p_2}dx\Big)^{1/p_2}\leq \epsilon\Big[F(\mu,u,\Om_{2R}(x_0))+\fint_{\Om_{2R}(x_0)}|Du|dx\Big].
	\end{equation}
\end{prop}
We have the following estimate.
\begin{prop}
	\label{prop3-boundary}
	For any $\epsilon > 0$ there exist $\delta>0$ and $R_\epsilon$ so that if $V$ be a weak solution to the problem \eqref{eq3s-boundary} with
	\begin{equation}\label{eq-ineq Vp2}
	\Big(\fint_{\Om_{R/2}}|DV|^{p_2}dx\Big)^{1/p_2}\leq \lambda,
	\end{equation}
	for some $\lambda\geq 1$ and $0<R<\f{R_{\epsilon}\wedge R_0\wedge R_\om\wedge R_a \wedge K_0^{-1}}{10}$, then there exists $v$ solving the problem \eqref{eq3-boundary} satisfying
		\begin{equation}\label{eq1-prop4}
		\|D\bar{v}\|_{L^{\vc}(B_{R/8}(x_0))}\lesi \lambda,
	\end{equation}
	and
	\begin{equation}
	\label{eq2-prop4}
	\fint_{\Om_{R/8}(x_0)}|D(\bar{v}-V)|^{p_2}\leq \epsilon^{p_2}\lambda
	\end{equation}
	where $\bar{v}$ is a zero extension of $v$ to $\Om_R$.
	\end{prop}
\begin{proof}
	
	We will show that there exists $v$ solving \eqref{eq3-boundary} satisfying
	\begin{equation}\label{eq1s-prop4}
	\|D\bar{v}\|_{L^{\vc}(B_{R/4})}\lesi \lambda,
	\end{equation}
	and
	\begin{equation}
	\label{eq2s-prop4}
	\fint_{\Om_{R/4}}|D(\bar{v}-V)|^{p_2}\leq \epsilon^{p_2}\lambda.
	\end{equation}
	Once these estimates are proved, the desired estimates in the proposition follows 
	immediately from the fact that $B_{R/8}(x_0)\subset B_{R/4}$.
	
Note that by using a suitable rescaling maps, it suffice to prove \eqref{eq1s-prop4} and \eqref{eq2s-prop4} with $R=8$ and $\lambda=1$.

We first prove that there exists $v$ solving \eqref{eq3-boundary} with
	\begin{equation}\label{eq1s-prop3}
	\fint_{\Om_4}|Dv|^{p_2}\leq 1
	\end{equation}
	satisfying
	\begin{equation}\label{eq2s-prop3}
	\fint_{\Om_4}|v-V|^{p_2}\leq \epsilon^{p_2}.
	\end{equation}
	To do this, observe  that if $V$ solves \eqref{eq3s-boundary} (with $R=8$), then it also solves
	\begin{equation}\label{eq2s-boundary}
	\left\{
	\begin{aligned}
	&\di \bar{b}_{B_R}(DV)=0 \quad &\text{in}& \quad \Om_{4},\\
	&V=0  \quad &\text{on}& \quad \partial_w \Om_{4}.
	\end{aligned}
	\right.
	\end{equation}
		We now proceed as in \cite{BW, BOR}. Assume, in the contrary, that there exist $\epsilon>0$ and the sequences $\{\Om_4^k\}_{k=1}^\vc$ and $\{V_k\}_{k=1}^\vc$ such that $V_k$ solves 
	\begin{equation}\label{eq2sk-boundary}
	\left\{
	\begin{aligned}
	&\di \bar{b}_{B_R}(DV_k)=0 \quad &\text{in}& \quad \Om^k_{4},\\
	&V_k=0  \quad &\text{on}& \quad \partial_w \Om^k_{4},
	\end{aligned}
	\right.
	\end{equation}
	where
\begin{equation}\label{eq-geometriccondition}
B_4^+\subset \Omega^k_4\subset \{x\in B_4: x_n>-24/k\}.
\end{equation}

	and
	$$
	\fint_{\Om_4^k}|DV_k|^{p_2}\leq 1.
	$$
	but
	\begin{equation}\label{eq-v-hk}
	\fint_{B_4^+}|v-V_k|^{p_2}\geq\epsilon
	\end{equation}
	for any weak solution $v$ to the equation
	\begin{equation}\label{eq3k-boundary}
	\left\{
	\begin{aligned}
	&\di \bar{b}_{B_4}(Dv)=0 \quad &\text{in}& \quad B_4^+,\\
	&v=0  \quad &\text{on}& \quad T_4,
	\end{aligned}
	\right.
	\end{equation}
	with $$
	\fint_{B_4^+}|Dv|^{p_2}\leq 1.
	$$
	
	Note that
	$$
	\fint_{B_4^+}|DV_k|^{p_2}\leq \fint_{\Om_4^k}|DV_k|^{p_2}\leq 1.
	$$
	As a consequence, there exists $V_0\in W^{1,p_2}(B_4^+)$ so that 
	$$
	V_k\to V_0 \ \ \text{strongly in $L^{p_2}(B_4^+)$}, \ \ \ DV_k\to DV_0 \ \ \text{weakly in $L^{p_2}(B_4^+)$}. 
	$$
	Moreover, it is not difficult to see that $h_0=0$ on $T_4$. So, by a straightforward manipulation, we infer that $h_0$ is a weak solution to the problem
	$$
	\left\{
	\begin{aligned}
		&\di \bar{b}_{B_4}(DV_0)=0 \quad &\text{in}& \quad B_4^+,\\
		&V_0=0  \quad &\text{on}& \quad T_4.
	\end{aligned}
	\right.
	$$
	This is a contradiction to \eqref{eq-v-hk} by taking $v=V_0$ and letting $k\to \vc$. Therefore, this proves \eqref{eq1s-prop3} and \eqref{eq2s-prop3}.

We now turn to prove \eqref{eq1-prop4} and \eqref{eq2-prop4}.

Since $v$ is a weak solution to \eqref{eq3-boundary} (with $R=8$), the H\"older regularity result implies that  
$$
\|D v\|^{p_2}_{L^\vc(B_2^+)}\leq  C\fint_{B_4^+}|Dv|^{p_2}dx\lesi 1,
$$
which implies \eqref{eq1-prop4}.

We now take care of \eqref{eq2-prop4}. To do this, we set 
$$
f(x)=-\chi_{\{x_n<0\}}\bar{b}^n_{B_4}(Dv(x',0)), g(x)=\chi_{\Om_{4}\backslash B_4^+}(x)\bar{b}_{B_4}(0),
$$
where $x=(x',x_n)$, and $\bar{b}_{B_4}=(\bar{b}^1_{B_4},\ldots,\bar{b}^n_{B_4})$. 

We now have the following lemma whose proof will be given after the proof of this proposition.
\begin{lem}\label{lem1}
Let  $\bar{v}$ be a zero extension of $v$ to $B_4$. Then  $\bar{v}$  solves the following equation
\begin{equation}\label{eq-bar v}
\left\{
\begin{aligned}
	&\di \bar{b}_{B_4}(D\bar{v})=D_nf +\di g\quad &\text{in}& \quad \Om_{4},\\
	&\bar{v}=0  \quad &\text{on}& \quad \partial_w \Om_{4}.
\end{aligned}
\right.
\end{equation}
\end{lem}

Let $\phi\in C^\vc_c(B_{4})$ satisfying $0\leq \phi\leq 1$, $\phi\equiv 1$ on $B_{2}$ and  supp\,$\phi\subset B_{3}$. Taking $\varphi=(h-\bar{v})\phi^{p_2}\in W^{1,p_2}_0(\Om_{4 })$, since $h$ and $\bar{v}$ are weak solutions to \eqref{eq2-boundary} and \eqref{eq-bar v},  we have
$$
\fint_{\Om_{4}}\langle \bar{b}_{B_4}(DV),D\varphi\rangle dx= 0,
$$
and
$$
\fint_{\Om_{4 }}\langle \bar{b}_{B_4}(D\bar{v}),D\varphi\rangle dx= \fint_{\Om_{4 }} fD_n\varphi dx+\fint_{\Om_{4}} gD\varphi dx.
$$
Hence,
$$
\begin{aligned}
\fint_{\Om_{4}}\langle \bar{b}_{B_4}(DV)-\bar{b}_{B_4}(D{\bar{v}}),D\varphi\rangle dx&= -\fint_{\Om_{4}} fD_n\varphi dx-\fint_{\Om_{4 }} gD\varphi dx.
\end{aligned}
$$

We consider two cases.

\noindent{\bf Case 1: $p_2\geq 2$}. 
Inserting $\varphi=(V-\bar{v})\phi^{p_2}$ into the above equation, then using \eqref{eq2-functionb} we have
\begin{equation}\label{eq-V-v}
\begin{aligned}
\fint_{\Om_{4}}\phi^{p_2}|D(V-\bar{v})|^{p_2}dx&\lesi \fint_{\Om_{4}}\phi^{p_2} \langle \bar{b}_{B_4}(DV)-\bar{b}_{B_4}(D{\bar{v}}),D(V-\bar{v})\rangle dx\\
&=-p_2\fint_{\Om_{4}}\phi^{p_2-1}(V-\bar{v})\langle \bar{b}_{B_4}(DV)-\bar{b}_{B_4}(D{\bar{v}}),D\phi\rangle dx\\
&\ \ \ \ -\fint_{\Om_{4}} fD_n\varphi dx-\fint_{\Om_{4}} gD\varphi dx\\
&=:I_1+I_2+I_3.
\end{aligned}
\end{equation}
For the term $I_1$, note that from \eqref{eq1-functionb} we have
$$
\begin{aligned}
|\bar{b}_{B_4}(DV)-\bar{b}_{B_4}(D{\bar{v}})|&\lesi (s^2+|DV|^2)^{\f{p_2-1}{2}}+(s^2+|D{\bar{v}}|^2)^{\f{p_2-1}{2}}\\
&\lesi |D(V-\bar{v})|^{p_2-1}+(1+|D{V}|)^{p_2-1}.
\end{aligned}
$$
Hence,
$$
\begin{aligned}
|I_1|\lesi \fint_{\Om_{4 }}\phi^{p_2-1}|V-\bar{v}|\left[|D(V-\bar{v})|^{p_2-1}+(1+|D{V}|)^{p_2-1}\right]|D\phi|dx.
\end{aligned}
$$
By Young's and H\"older's inequalities, we infer that 
$$
\begin{aligned}
|I_1|&\leq \tau\fint_{\Om_{4}}\phi^{p_2}|D(V-\bar{v})|^{p_2}dx\\
&\ \ \ \ +c(\tau)\fint_{\Om_{4 }}|V-\bar{v}|^{p_2}dx+c\Big(\fint_{\Om_{4 }}|V-\bar{v}|^{p_2}dx\Big)^{1/p_2}\Big(\fint_{\Om_{4 }}(1+|DV|)^{p_2}dx\Big)^{\f{p_2-1}{p_2}}\\
&\leq \tau\fint_{\Om_{4}}\phi^{p_2}|D(V-\bar{v})|^{p_2}dx+c(\tau)\fint_{\Om_{4 }}|V-\bar{v}|^{p_2}dx+c\Big(\fint_{\Om_{4 }}|V-\bar{v}|^{p_2}dx\Big)^{1/p_2}.
\end{aligned}
$$
If $2\leq p_2<n$, then using H\"older's inequality, Sobolev inequality, \eqref{eq-ineq Vp2} and \eqref{eq2s-prop3} we have
\begin{equation}\label{eq p2<n}
\begin{aligned}
\fint_{\Om_{4 }}|V-\bar{v}|^{p_2}dx&\leq \fint_{B_4^+}|V-v|^{p_2}+\fint_{\Om_{4 }\backslash B_4^+}|V|^{p_2}dx\\
&\leq c\epsilon^{p_2}+ \Big( \int_{\Om_{4 }\backslash B_4^+}|V|^{\f{np_2}{n-p_2}}dx\Big)^{\f{n-p_2}{n}}\f{|\Om_{4 }\backslash B_4^+|^{p_2/n}}{|\Om_{4 }|}\\
&\leq c\epsilon^{p_2}+ \delta^{p_2/n}\int_{\Om_{4 }\backslash B_4^+}|DV|^{p_2}dx\\
&\leq c(\epsilon^{p_2}+\delta^{p_2/n}).
\end{aligned}
\end{equation}
If $p\geq n$, similarly, we have
\begin{equation}\label{eq p2>n}
\begin{aligned}
\fint_{\Om_{4 }}|V-\bar{v}|^{p_2}dx&\leq \fint_{B_4^+}|V-v|^{p_2}+\fint_{\Om_{4 }\backslash B_4^+}|V-\bar{v}|^{p_2}dx\\
&\leq c\epsilon^{p_2}+ \Big( \int_{\Om_{4 }\backslash B_4^+}|V|^{2p_2}dx\Big)^{1/2}\f{|\Om_{4 }\backslash B_4^+|^{1/2}}{|\Om_{4 }|}\\
&\leq c\epsilon^{p_2}+ \delta^{1/2}\int_{\Om_{4 }\backslash B_4^+}|DV|^{p_2}dx\\
&\leq c(\epsilon^{p_2}+\delta^{1/2}).
\end{aligned}
\end{equation}
Let us estimate the term $I_2$. We have
$$
\begin{aligned}
|I_2|&\leq \fint_{\Om_{4}} |\phi^{p_2}f D_n (V-\bar{v})+p_2\phi^{p_2-1} (V-\bar{v})f D_n\phi|  dx\\
&\leq \f{|\Om_{4}\backslash B_4^+|}{|\Om_{4}|}\fint_{\Om_{4}\backslash B_4^+} |\phi^{p_2}f D_n (V-\bar{v})+p_2\phi^{p_2-1} (V-\bar{v})f D_n\phi|  dx\\
&\leq c\delta\fint_{\Om_{4}\backslash B_4^+} \sum_{i=1}^{n-1}(1+ |Dv(x',0)|)^{p_2-1}\left[|\phi^{p_2}DV|+| V | \right] dx.
\end{aligned}
$$	
This along with \eqref{eq1-prop4} yields
$$
\begin{aligned}
|I_2|
&\leq c\delta\fint_{\Om_{4}\backslash B_4^+}|\phi^{p_2}DV|+| V | dx\\
&\leq c\delta\fint_{\Om_{4}\backslash B_4^+}|\phi^{p_2}DV|+| V | dx.
\end{aligned}
$$
The  H\"older's inequality and \eqref{eq-ineq Vp2} imply that
$$
\fint_{\Om_{4}\backslash B_4^+}|\phi^{p_2}DV|dx\leq \Big(\fint_{\Om_{4}\backslash B_4^+}|\phi^{p_2}DV|^{p_2}dx\Big)^{1/p_2}\leq \Big(\fint_{\Om_{4}\backslash B_4^+}|DV|^{p_2}dx\Big)^{1/p_2}\leq C.
$$
Moreover, arguing similarly to \eqref{eq p2<n} and \eqref{eq p2>n} we have
$$
\fint_{\Om_{4}\backslash B_4^+}| V | dx\leq C.
$$ 
Hence,
$$
|I_2|\leq C\delta.
$$

Finally, by using \eqref{eq1-functionb} it can be verified that
$$
|I_3|\leq C|\Om_{4}\backslash B_4^+|\leq C\delta.
$$

Inserting the estimates $I_1, I_2$ and $I_3$ into 
\eqref{eq-V-v}, we obtain
$$
\begin{aligned}
\fint_{\Om_{4}}\phi^{p_2}|D(V-\bar{v})|^{p_2}dx\leq \mathcal{O}(\epsilon)+\mathcal{O}(\delta)+\tau \fint_{\Om_{4}}\phi^{p_2}|D(V-\bar{v})|^{p_2}dx.
\end{aligned}
$$
Taking $\tau<1$ we get
$$
\begin{aligned}
\fint_{\Om_{4}}\phi^{p_2}|D(V-\bar{v})|^{p_2}dx\leq \mathcal{O}(\epsilon)+\mathcal{O}(\delta),
\end{aligned}
$$
which deduces \eqref{eq2-prop4}.

\bigskip

\noindent{\bf Case 2: $2-1/n<p_2<2$.} By the standard argument as \eqref{eq2s-proof prop1 inter}, we also obtain, for $\tau>0$, 
$$
\begin{aligned}
\fint_{\Om_{4}}\phi^{p_2}|D(V-\bar{v})|^{p_2}dx\leq&  \tau\fint_{\Om_{4}}\phi^{p_2}|D(V-\bar{v})|^{p_2}dx +\tau \Big(\fint_{\Om_{R}}|DV|^{p_2}dx+1\Big)\\
& + c(\tau)\fint_{\Om_{4}}\phi^{p_2} \langle \bar{b}_{B_4}(DV)-\bar{b}_{B_4}(D{\bar{v}}),D(V-\bar{v})\rangle dx.
\end{aligned}
$$
At this stage, repeating the argument above we can prove
$$
\fint_{\Om_{4}}\phi^{p_2} \langle \bar{b}_{B_4}(DV)-\bar{b}_{B_4}(D{\bar{v}}),D(V-\bar{v})\rangle dx\leq \mathcal{O}(\epsilon)+\mathcal{O}(\delta).
$$
This along with \eqref{eq-ineq Vp2} yields,
$$
\begin{aligned}
\fint_{\Om_{4}}\phi^{p_2}|D(V-\bar{v})|^{p_2}dx\leq&  \tau\fint_{\Om_{4}}\phi^{p_2}|D(V-\bar{v})|^{p_2}dx +c\tau+\mathcal{O}(\epsilon)+\mathcal{O}(\delta).
\end{aligned}
$$
By taking $\tau$ sufficiently small, we obtain
$$
\begin{aligned}
\fint_{\Om_{4}}\phi^{p_2}|D(V-\bar{v})|^{p_2}dx\leq \mathcal{O}(\epsilon)+\mathcal{O}(\delta),
\end{aligned}
$$
which implies \eqref{eq2-prop4}.
\end{proof}

We now give the proof of Lemma \ref{lem1}.
\begin{proof}
	[Proof of Lemma \ref{lem1}:] We use some ideas in \cite{Ph}. For the sake of simplicity we denote $\bar{b}_{B_5}$ by ${\bf b}=({\bf b}^1,\ldots,{\bf b}^n)$. 	Let $\varphi\in C_0^\vc(\Om_{4})$ and let $\ell\in C_c^\vc(\mathbb{R})$ satisfying $0\leq \ell\leq 1$, $\ell(t)=1, t\geq 1$, supp\,$h\subset [1/2,\vc)$. Then for each $\epsilon>0$, we set   $\phi_\epsilon(x)=\varphi(x)\ell\Big(\f{x_n}{\epsilon}\Big)\in W^{1,p_2}_0(B_4^+)$.
	
	Therefore, taking $\phi_\epsilon(x)$ as a test function we have
	$$
	\begin{aligned}
	\int_{B_4^+}\langle{\bf b}(Dv),D\phi_\epsilon\rangle dx=0.
	\end{aligned}
	$$
	This implies that
	$$
	\begin{aligned}
	\int_{B_4^+}\ell\Big(\f{x_n}{\epsilon}\Big) \langle{\bf b}(Dv),D\varphi\rangle dx
	=&-\int_{B_4^+}{\bf b}^n(Dv)\ell'\Big(\f{x_n }{\epsilon}\Big)\varphi(x)\f{dx}{\epsilon}\\
	=&-\int_{0}^{4}\int_{|x'|<\sqrt{4^2-|x_n|^2}}{\bf b}^n(Dv)\varphi(x)dx'\ell'\Big(\f{x_n }{\epsilon}\Big)\f{dx_n}{\epsilon}\\
	\end{aligned}
	$$
	Letting $\epsilon\to 0$ and using the integration by part, we obtain
	$$
	\begin{aligned}
	\int_{B_4^+}\langle{\bf b}(Dv),D\varphi\rangle dx
	&=-\int_{|x'|<4}{\bf b}^n(Dv(x',0))\varphi(x',0)dx'\\
	&=-\int_{B_4^-}{\bf b}^n(Dv(x',0))D_n\varphi(x)dx
	\end{aligned}
	$$
	where $B_4^-=B_4\cap \{x_n<0\}$.
	
	This together with the fact that $D\bar{v}=0$ in $B_4^-$ implies that
	$$
	\begin{aligned}
	\int_{\Om_{4}}\langle{\bf b}(D\bar{v}),D\varphi\rangle dx&=\int_{B_4^+}\ldots+\int_{\Om_4\backslash B_4^+}\ldots\\
	&=\int_{B_4^+}f(x)D_n\varphi(x)dx+\int_{\Om_4\backslash B_4^+}\langle{\bf b}(0),D\varphi\rangle dx\\
	&=\int_{\Om_{4}}f(x)D_n\varphi(x)dx+\int_{\Om_{4}}\langle g(x),D\varphi\rangle dx.
	\end{aligned}
	$$
	This completes the proof.
\end{proof}

We have the following corollary.
\begin{cor}
		\label{cor-boundary}
		For any $\epsilon>0$ there exist $\delta>0$ and $R_\epsilon$ so that  if $u$ is a weak solution to \eqref{Maineq} with
		\begin{equation}\label{eq1-cond Du}
		\fint_{\Om_{2R}(x_0)}|Du|dx\leq \lambda,
		\end{equation}
		and
		\begin{equation}\label{eq1-cond F}
		F(\mu,u,\Om_{2R}(x_0))\leq \delta\lambda,
		\end{equation}
		for some $\lambda\geq 1$ and $R<\f{R_\epsilon\wedge R_0\wedge R_\om\wedge R_a\wedge K_0^{-1}}{10}$, then there exists $v\in W^{1,p_2}(B_{R/8}(x_0))\cap W^{1,\vc}(B_{R/8}(x_0))$ satisfying
		\begin{equation}\label{eq1-cor}
		\|Dv\|^{p_2}_{L^{\vc}(B_{R/8}(x_0))}\lesi  \lambda^{p_2},
		\end{equation}
		and
		\begin{equation}
		\label{eq2-cor}
				\fint_{\Om_{R/8}(x_0)}|D(u-v)|^{p_2}\leq \epsilon^{p_2} \lambda^{p_2}.
		\end{equation}
				\end{cor}
\begin{proof}
	Let $w, h, V, v$ be solutions of the problems as above. For the sake of simplicity we still denote by $v$ the zero extension of $v$ to $B_{R/8}(x_0)$.
	It is clear that
	$$
	\begin{aligned}
	\fint_{\Om_{R/2}}|DV|^{p_2}dx&\lesi \fint_{\Om_R(x_0)}|D(V-h)|^{p_2}dx+\fint_{\Om_R(x_0)}|D(h-w)|^{p_2}dx+\fint_{\Om_R(x_0)}|Dw|^{p_2}dx\\
	&\lesi \Big(\fint_{\Om_{2R}(x_0)}|Du|dx\Big)^{p_2} + F(\mu,u,\Om_{2R}(x_0))^{p_2}+\fint_{\Om_R(x_0)}|Dw|^{p_2}dx\\
	&\lesi \lambda^{p_2}+\fint_{\Om_R(x_0)}|Dw|^{p_2}dx,
	\end{aligned}
	$$
	where in the second inequality we used Propositions \ref{prop2-boundary} and \ref{prop3-boundary}, and in the last inequality we used \eqref{eq1-cond Du} and \eqref{eq1-cond F}.
	
	On the other hand, arguing similarly to the proof of \eqref{eq2-estimate p2}, we obtain
	$$
	\fint_{\Om_R(x_0)}|Dw|^{p_2}dx\lesi \Big(\fint_{\Om_{2R}(x_0)}|Dw|dx\Big)^{p_2},
	$$
	which along with Proposition \ref{prop1-boundary}, \eqref{eq1-cond Du} and \eqref{eq1-cond F} yields
	$$
	\fint_{\Om_R(x_0)}|Dw|^{p_2}dx\lesi \lambda^{p_2}.
	$$
	Hence,
	$$
	\fint_{\Om_{R/2}}|DV|^{p_2}dx\lesi \lambda^{p_2}.
	$$
		At this stage the desired estimates follow directly Proposition \ref{prop3-boundary}.
\end{proof}
\section{Weighted regularity estimates}
We fix $w\in A_\vc$, $0<\sigma_0<\min\left\{\f{n(\gamma_1-1)}{n-1},n\right\}$ and $x_0\in \overline{\Om}$. 

We set $a_0=\f{1}{(c_2+c_6+1)c_w}$, where $c_2, c_6$ are constant determined as in the proof of Theorem \ref{thm-goodlamdbda} and $c_w$ is a constant in Lemma \ref{weightedlemma2}. Taking 
\begin{equation}
\label{eq-epsilon0}
\epsilon_0=\f{(BA_0^{\gamma_4})^{-\kappa_w}}{a_0},
\end{equation}

where $A_0, B$ are constants defined in Theorem \ref{thm-goodlamdbda} below.

Set $\epsilon=(a_0 \epsilon_0)^{1/\kappa_w}$. We now fix a number $R$
\begin{equation}\label{eq-defn R}
R:= \f{R_0\wedge R_\om\wedge R_{\nu}\wedge R_{\epsilon} \wedge R_a\wedge K_0^{-1}}{20},
\end{equation}
where $R_\epsilon$ is a constant determined as in Proposition \ref{prop3-interior} and Corollary \ref{cor-boundary}, and $R_\nu$ is a constant satisfying
\begin{equation}\label{eq-rnu}
\nu(r)\log\Big(\f{1}{r}\Big)\leq \f{1}{2} \ \ \text{and} \ \ \  \nu(80r)<\min\{\gamma_3\sigma_0,\gamma_3(\gamma_1-1)\}, \ \ \ \ \text{for all} \ \ \ r<R_\nu.
\end{equation}
We set 
$$
q_-=\inf_{x\in \Om_{2R}(x_0)}q(x), \ \ \ q_+=\sup_{x\in \Om_{2R}(x_0)}q(x).
$$

We now define
\begin{equation}\label{defn-M epsilon}
M(\sigma_0,\epsilon,\Om_{2R},w)=\f{1}{\epsilon}\fint_{\Om_{2R}}|Du|^{1+\sigma_0} dx.
\end{equation}

	The following result is taken from \cite{MP} which can be seen as an extension of a variant Vitali covering lemma in \cite{CP} to the weighted case.	
\begin{lem}\label{lem2-proofmainresutls}
	Let $\Om$ be a $(\delta, R_0)$ Reifenberg flat domain, $w\in A_\vc$ and $R<R_0$. Suppose that $E\subset G\subset \Om_{R}\equiv \Om_{R}(x_0)$ are measurable and satisfy the following conditions:
	\begin{enumerate}[{\rm (a)}]
		\item $w(E)< \epsilon_0 
		\Om_{R}$;
		\item for any ball $B_\rho(y)$ with $\rho\in (0,R)$ and $y\in E$, if $w(E\cap B_\rho(y))\geq \epsilon_0 w(B_\rho(y))$ then $\Om_{R}\cap B_\rho(y)\subset G$.
	\end{enumerate}
	Then there exists $c=c(n,w)$  such that
	$$
	w(E)\leq c\epsilon_0 w(G).
	$$
\end{lem}

We now prove the good $\lambda$-inequality which plays a key role in the proofs  of our main results.
\begin{thm}\label{thm-goodlamdbda}
Let $w\in A_\vc$ and   let $u$ be a weak solution to \eqref{Maineq}. Then there exists $A_0=A_0(n,\Lambda_1,\Lambda_2)>1$ so that the following holds true. For any $R_0>0$, there exists $\delta=\delta(n,\Lambda_1,\Lambda_2,\epsilon_0,w)$ such that if $\Om$ is a $(\delta,R_0)$ Reifenberg domain and the nonlinearity $a$ satisfies \eqref{eq1-functiona}, \eqref{eq1s-functiona} and the small BMO norm condition \eqref{eq2-Assumption}, then for all $\lambda>0$,
$$
\begin{aligned}
w\Big(\Big\{x\in \Om_{R}:\mathcal{M}(|Du|^{\f{q(\cdot)}{q_-}}\chi_{\Om_{2R}})>A_0\lambda,  \Mu^{\f{1}{\pd-1}\f{q(\cdot)}{q_-}} +1\leq \lambda\Big\}\Big)\\
\leq B\epsilon w\Big(\Big\{x\in \Om_{R}:\mathcal{M}(|Du|^{\f{q(\cdot)}{q_-}}&\chi_{\Om_{2R}})>\lambda\Big\}\Big),
\end{aligned}
$$ 
where $B$ is a constant independing on $\epsilon_0$.
\end{thm}
	
	\begin{proof}
		[Proof of Theorem \ref{thm-goodlamdbda}:]
		We set 
		$$
		E:=\left(\left\{x\in \Om_{R}:\mathcal{M}(|Du|^{\f{q(\cdot)}{q_-}}\chi_{\Om_{2R}})>A_0\lambda, \Mu^{\f{1}{\pd-1}\f{q(\cdot)}{q_-}} \leq \alpha \lambda\right\}\right),
		$$
		and 
		$$
		G:=\left(\left\{x\in \Om_{R}:\mathcal{M}(|Du|^{\f{q(\cdot)}{q_-}}\chi_{\Om_{2R}})> \lambda\right\}\right)
		$$

		Since the Hardy-Littlewood maximal function $\mathcal{M}$ is weak type $(1,1)$ and 
		$$
		\f{q(x)}{q_-}\leq \f{q_-+\nu(4R)}{q_-}\leq \f{\gamma_3 +\nu(4R)}{\gamma_3}\leq 1+\sigma_0, \ \forall x\in \Om_{2R}, 
		$$
		we have
		$$
		\begin{aligned}
		|E|&\leq \f{c(n)}{A_0\lambda}\int_{\Om_{2R}}|Du|^{\f{q(x)}{q_-}}dx\leq \f{c(n)}{A_0\lambda}\int_{\Om_{2R}}|Du|^{1+\sigma_0}+1 \ dx\\
		&\leq \f{c(n)}{A_0\lambda}|\Om_{2R}|M(\epsilon,\sigma_0,\Om_{2R}).
		\end{aligned}
		$$
		
		Hence, we have
		$$
		|E|\leq \epsilon|\Om_{2R}|= (a_0\epsilon_0)^{1/\kappa_w} |\Om_{2R}|,
		$$
		as long as $\lambda>c(n)M(\epsilon_0,\sigma_0, \Om_{2R})$.
		
		This together with Lemma \ref{weightedlemma2} implies that
		$$
		w(E)\leq  c_wa_0\epsilon_0w(\Om_{R})\leq \epsilon_0w(\Om_{R}). 
		$$
		
		We now verify the condition {\rm (b)} in Lemma \ref{lem2-proofmainresutls}. To do this we argue by contradiction. 
		Indeed, assume that $\Om_{R}\cap B_{\rho}(y_0)\cap G^c\neq \emptyset$ for some $y_0\in \Om_{R}$ and $\rho\in (0, R)$. 
		Due to Lemma \ref{weightedlemma2}, it suffices to prove that
		\begin{equation}\label{eq1-proof thm goodlambda}
		w(E\cap B_{\rho}(y_0))< \epsilon_0 w(B_{\rho}(y_0)).
		\end{equation}
		Let $x_1\in B_{\rho}(y_0)\cap G^c$ and $x_2\in E\cap B_{\rho}(y_0)$. Hence, we have, for any $r>0$,
		\begin{equation}\label{eq2-proof thm goodlambda}
		\mathcal{M}(|Du|^{\f{q(\cdot)}{q_-}}\chi_{\Om_{2R}})(x_1)\leq \lambda, \ \ \ \Mu (x_2)^{\f{1}{p(x_2)-1}\f{q(x_2)}{q_-}}+1\leq \lambda. 
		\end{equation}
		We now consider two cases:  $B_{4\rho}(y_0)\cap \Omega_{2R}\neq \emptyset$ and $B_{4\rho}(y_0)\subset \Omega_{2R}$.

		We just consider the case $B_{4\rho}(y_0)\cap \Omega_{2R}\neq \emptyset$, since the case  $B_{4\rho}(y_0)\subset \Omega_{2R}$ can be argued similarly. Since $y_0\in \Om_{R}$ and $\rho<\f{R}{100}$,  there exists $z_0\in B_{4\rho}(y_0)\cap \partial_w \Omega_{2R}$.
		
		Set $$
		\tilde{q}_-:=\inf_{x\in B_{40\rho}(y_0)}q(x), \ \ \ \text{and} \ \ \ \tilde{q}_+:=\sup_{x\in B_{40\rho}(y_0)}q(x).
		$$
		To deal with this case we need the following result whose proof will be given later.
		
        \begin{lem}
        	\label{lem1-proof}
        We have
		\begin{equation}\label{eq1-proof lemma good lambda}
		\fint_{\Om_{40\rho}(z_0)}|Du|dx\lesi  \lambda^{\f{q_-}{\tilde{q}_+}},
		\end{equation}
		and
		\begin{equation}\label{eq2-proof lemma good lambda}
		\inf_{x\in \Om_{40\rho}(z_0)}\Big[\f{|\mu|(\Om_{40\rho}(x))}{\rho^{n-1}}\Big]^{\f{1}{p(x)-1}}+1\lesi \lambda^{\f{q_-}{\tilde{q}_+}}.
		\end{equation}
		\end{lem}	
		
		Applying Corollary \ref{cor-boundary} we can find $v\in W^{1,p_2}(B_{5\rho}(z_0))\cap W^{1,\vc}(B_{5\rho}(z_0))$ so that 
		\begin{equation}\label{eq1-proof good lambda}
		\|Dv\|_{L^{\vc}(\Om_{5\rho}(z_0))}\leq c_4 \lambda^{\f{q_-}{\tilde{q}_+}},
		\end{equation}
		and
		\begin{equation}\label{eq2-proof good lambda}
		\Big(\fint_{\Om_{5\rho}(z_0)}|D(v-u)|^{p_2}dx\Big)^{1/p_2}\leq c_5\epsilon \lambda^{\f{q_-}{\tilde{q}_+}}.
		\end{equation}
		Since $\f{q(x)}{q_-}\leq \f{\tilde{q}_+}{q_-}$ for $x\in B_{\rho}(y_0)\subset B_{40\rho}(z_0)$, we have
	    $$
	    |Du|^{\f{q(x)}{q_-}}\leq |D(u-v)|^{\f{\tilde{q}_+}{q_-}}+(|Dv|^{\f{\tilde{q}_+}{q_-}}+1), 
	    $$
	    for all $x\in B_{\rho}(y_0)$.
	    
	    Choosing $A_0= \max{2c_4 +1}$ then we have
	    $$
	    \begin{aligned}
	    |E\cap B_{\rho}(y_0))|&\leq \left\{x\in B_{\rho}(y_0):\mathcal{M}(|D(u-v)|^{\f{\tilde{q}_+}{q_-}}\chi_{\Om_{2R}})>A_0\lambda/2\right\}\\
	    &+\left\{x\in B_{\rho}(y_0):\mathcal{M}(|Dv|^{\f{\tilde{q}_+}{q_-}}+1\chi_{\Om_{2R}})>A_0\lambda/2\right\}.	    
	    \end{aligned}
	    $$
	    Note that the second term on the right hand side of the inequality above is zero due to \eqref{eq1-proof good lambda} and $A_0= \max\{2c_4 +1\}$. Theorefor, by the weak type $(1,1)$ of the Hardy-Littlewood maximal function, we have 
	    $$
	    \begin{aligned}
	    |E\cap B_{\rho}(y_0)|&\leq \left\{x\in B_{\rho}(y_0):\mathcal{M}(|D(u-v)|^{\f{\tilde{q}_+}{q_-}}\chi_{\Om_{2R}})>A_0\lambda/2\right\}\\
	    &\leq \f{c}{A_0\lambda} \int_{B_{\rho}(y_0)}|D(u-v)|^{\f{\tilde{q}_+}{q_-}}\\
	    &\leq \f{cB_{\rho}(y_0)}{A_0\lambda} \fint_{B_{5\rho}(z_0)}|D(u-v)|^{\f{\tilde{q}_+}{q_-}}\\
	    \end{aligned}
	    $$
	    From \eqref{eq-rnu}, $\f{\tilde{q}_+}{q_-}<\gamma_1\leq p_2$. Using H\"older's inequality, we obtain
	    $$
	    \begin{aligned}
	    |E\cap B_{\rho}(y_0)|
	    &\leq \f{cB_{\rho}(y_0)}{A_0\lambda} \Big(\fint_{B_{5\rho}(z_0)}|D(u-v)|^{p_2}\Big)^{\f{\tilde{q}_+}{p_2q_-}}\\
	    &\leq \f{c|B_{\rho}(y_0)|}{A_0\lambda} \epsilon^{\f{\tilde{q}_+}{q_-}}\lambda=c_6(a_0\epsilon_0)^{\f{\tilde{q}_+}{\kappa_wq_-}}|B_{\rho}(y_0)|,
	    \end{aligned}
	    $$
	    where in the last inequality we used \eqref{eq2-proof good lambda}.

	    This in combination with Lemma \ref{weightedlemma2} yields
	    $$
	    \begin{aligned}
	    w(E\cap B_{\rho}(y_0))&\leq c_w \Big[c_6(a_0\epsilon_0)^{\f{\tilde{q}_+}{\kappa_wq_-}}\Big]^{\kappa_w}w(B_{\rho}(y_0))\\
	    &\leq c_wc_6^{\kappa}a_0\epsilon_0 w(B_{\rho}(y_0))\\
	    &\leq \epsilon_0 w(B_{\rho}(y_0)).
	    \end{aligned}
	    $$
	    This gives \eqref{eq1-proof thm goodlambda}. The proof is complete.
	\end{proof}

	We now prove Lemma \ref{lem1-proof}.
	
	\begin{proof}[Proof of Lemma \ref{lem1-proof}:]
		We have, by H\"older's inequality,
		$$
		\begin{aligned}
		\fint_{\Om_{40\rho}(z_0)}|Du|dx&= \Big(\fint_{\Om_{40\rho}(z_0)}|Du|dx\Big)^{\f{\tilde{q}_-}{\tilde{q}_+}}\Big(\fint_{\Om_{40\rho}(z_0)}|Du|dx\Big)^{1-\f{\tilde{q}_-}{\tilde{q}_+}}\\
		&\leq \Big(\fint_{\Om_{40\rho}(z_0)}|Du|^{\f{\tilde{q}_-}{q_-}}dx\Big)^{\f{q_-}{\tilde{q}_+}}\Big(\fint_{\Om_{40\rho}(z_0)}|Du|dx\Big)^{1-\f{\tilde{q}_-}{\tilde{q}_+}}\\
		&\leq \Big(\fint_{\Om_{40\rho}(z_0)}|Du|^{\f{q(x)}{q_-}}+1 \ \ dx\Big)^{\f{q_-}{\tilde{q}_+}}\Big(\fint_{\Om_{40\rho}(z_0)}|Du|dx\Big)^{\f{\tilde{q}_+-\tilde{q}_-}{\tilde{q}_+}}.
		\end{aligned}
		$$
		From \eqref{eq2-proof thm goodlambda}, we have
		$$
		\Big(\fint_{\Om_{40\rho}(z_0)}|Du|^{\f{q(x)}{q_-}}+1 \ \ dx\Big)^{\f{q_-}{\tilde{q}_+}}\lesi \lambda^{\f{q_-}{\tilde{q}_+}}.
		$$
		Moreover, by \eqref{bounds-Du} and \eqref{cond2-px} we have
		$$
		\begin{aligned}
		\Big(\fint_{\Om_{40\rho}(z_0)}|Du|dx\Big)^{1-\f{\tilde{q}_-}{\tilde{q}_+}}&\leq |\Om_{40\rho}(z_0)|^{\f{\tilde{q}_+-\tilde{q}_-}{\tilde{q}_+}}(1+\mu(\Om)^{\f{1}{\gamma_2-1}})^\f{\tilde{q}_+-\tilde{q}_-}{\tilde{q}_+}\\
		&\leq c(n,\mu,\Om, p(\cdot)).
		\end{aligned}
		$$
		As a consequence, we obtain \eqref{eq1-proof lemma good lambda}.
		
		\medskip
		
        In order to prove \eqref{eq2-proof lemma good lambda}, we write
        $$
        \begin{aligned}
        \fint_{\Om_{40\rho}(z_0)}\Big[\f{|\mu|(\Om_{40\rho}(x))}{\rho^{n-1}}\Big]^{\f{1}{p(x)-1}}+1 \ \ dx\leq& \Big(\fint_{\Om_{40\rho}(z_0)}\Big[\f{|\mu|(\Om_{40\rho}(x))}{\rho^{n-1}}\Big]^{\f{1}{p(x)-1}}+1 \ \ dx\Big)^{\f{\tilde{q}_-}{\tilde{q}_+}}\\
        &\ \ \times\Big(\fint_{\Om_{40\rho}(z_0)}\Big[\f{|\mu|(\Om_{40\rho}(x))}{\rho^{n-1}}\Big]^{\f{1}{p(x)-1}}+1 \ \ dx\Big)^{\f{\tilde{q}_+-\tilde{q}_-}{\tilde{q}_+}}.
        \end{aligned}
        $$	
        Arguing similarly as above, we can prove
        $$
        \Big(\fint_{\Om_{40\rho}(z_0)}\Big[\f{|\mu|(\Om_{40\rho}(x))}{\rho^{n-1}}\Big]^{\f{1}{p(x)-1}}+1 \ \ dx\Big)^{\f{\tilde{q}_+-\tilde{q}_-}{\tilde{q}_+}}\leq c(n,\mu,\Om,p(\cdot)).
        $$	
        This together with H\"older's inequality implies
        $$
        \begin{aligned}
        \fint_{\Om_{40\rho}(z_0)}\Big[\f{|\mu|(\Om_{40\rho}(x))}{\rho^{n-1}}\Big]^{\f{1}{p(x)-1}}+1 \ \ dx\leq& C\Big(\fint_{\Om_{40\rho}(z_0)}\Big[\f{|\mu|(\Om_{40\rho}(x))}{\rho^{n-1}}\Big]^{\f{1}{p(x)-1}\f{\tilde{q}_-}{q_-}}+1 \ \ dx\Big)^{\f{q_-}{\tilde{q}_+}}\\
        \leq& C\Big(\fint_{\Om_{40\rho}(z_0)}\Big[\f{|\mu|(\Om_{40\rho}(x))}{\rho^{n-1}}\Big]^{\f{1}{p(x)-1}\f{q(x)}{q_-}}+1 \ \ dx\Big)^{\f{q_-}{\tilde{q}_+}}\\
        \leq& \alpha^{\f{q_-}{\tilde{q}_+}} \lambda^{\f{q_-}{\tilde{q}_+}} 
        \end{aligned}
        $$
	\end{proof}
	
	\begin{proof}[Proof of Theorem \ref{mainthm1}:]
		$$
		\begin{aligned}
		\Big\|\mathcal{M}(|Du|^{\f{q(\cdot)}{q_-}}\chi_{\Om_{2R}})\Big\|^{q_-}_{L^{q_-}_w(\Om_{R})}&=q_-\int_0^\vc (A_0\lambda)^{q_-}w\Big(\Big\{x\in \Om_{R}:\mathcal{M}(|Du|^{\f{q(\cdot)}{q_-}}\chi_{\Om_{2R}})>A_0\lambda\Big\}\Big)\f{d\lambda}{\lambda}\\
		&=q_-\int_0^{\lambda_0} (A_0\lambda)^{q_-}w\Big(\Big\{x\in \Om_{R}:\mathcal{M}(|Du|^{\f{q(\cdot)}{q_-}}\chi_{\Om_{2R}})>A_0\lambda\Big\}\Big)\f{d\lambda}{\lambda}\\
		& \ \ \ \ + q_-\int_{\lambda_0}^\vc (A_0\lambda)^{q_-}w\Big(\Big\{x\in \Om_{R}:\mathcal{M}(|Du|^{\f{q(\cdot)}{q_-}}\chi_{\Om_{2R}})>A_0\lambda\Big\}\Big)\f{d\lambda}{\lambda}\\&=: I_1 + I_2
		\end{aligned}
		$$
		where $\lambda_0=M(\sigma_0,\epsilon,\Om_{2R},w)$ defined in \eqref{defn-M epsilon}.
		
		For $I_1$ we have
		$$
		\begin{aligned}
		I_1\leq q_-w(\Om_{R})\lambda_0^{q_-}\leq c\Big(\f{1}{\epsilon}\fint_{\Om_{2R}}|Du|^{1+\sigma_0} + 1 \ dx\Big)^{q_-}.
		\end{aligned}
		$$
		
		To estimate $I_2$, by Theorem \ref{thm-goodlamdbda} we have
		$$
		\begin{aligned}
	    I_2&\leq q_-\int_{\lambda_0}^\vc (A_0\lambda)^{q_-}w\Big(\Big\{x\in \Om_{R}:\Mu^{\f{1}{\pd-1}\f{q(\cdot)}{q_-}} >   \lambda\Big\}\Big)\f{d\lambda}{\lambda}\\
	    & + B\epsilon q_-\int_{\lambda_0}^\vc (A_0\lambda)^{q_-}w\Big(\Big\{x\in \Om_{R}:\mathcal{M}(|Du|^{\f{q(\cdot)}{q_-}}\chi_{\Om_{2R}})> \lambda\Big\}\Big)\f{d\lambda}{\lambda}\\
	    &:=A_0^{q_-} \Big\|\Mu^{\f{1}{\pd-1}\f{q(\cdot)}{q_-}}\Big\|_{L^{q_-}_w(\Om_{R})}+BA_0^{q_-} \epsilon\Big\|\mathcal{M}(|Du|^{\f{q(\cdot)}{q_-}}\chi_{\Om_{2R}})\Big\|^{q_-}_{L^{q_-}_w(\Om_{R})}.
	    \end{aligned}
	    $$
	    From \eqref{eq-epsilon0}, we have $BA_0^{q_-} \epsilon<1/2$ and hence,
	    $$
	    \begin{aligned}
	    \Big\|\mathcal{M}(|Du|^{\f{q(\cdot)}{q_-}}\chi_{\Om_{2R}})\Big\|^{q_-}_{L^{q_-}_w(\Om_{R})}\leq& c\Big(\f{1}{\epsilon}\fint_{\Om_{2R}}|Du|^{1+\sigma_0} + 1 \ dx\Big)^{q_-}
	    +A_0^{q_-} \Big\|\Mu^{\f{1}{\pd-1}\f{q(\cdot)}{q_-}}\Big\|_{L^{q_-}_w(\Om_{R})}\\
	    & \ \ \ +\f{1}{2}\Big\|\mathcal{M}(|Du|^{\f{q(\cdot)}{q_-}}\chi_{\Om_{2R}})\Big\|^{q_-}_{L^{q_-}_w(\Om_{R})}.
	    \end{aligned}
	    $$
	    This implies
	    $$
	    \Big\|\mathcal{M}(|Du|^{\f{q(\cdot)}{q_-}}\chi_{\Om_{2R}})\Big\|^{q_-}_{L^{q_-}_w(\Om_{R})}\leq cR^{-nq_-}\Big(\f{1}{\epsilon}\int_{\Om_{2R}}|Du|^{1+\sigma_0} + 1 \ dx\Big)^{q_-}
	    +2A_0^{q_-} \Big\|\Mu^{\f{1}{\pd-1}\f{q(\cdot)}{q_-}}\Big\|_{L^{q_-}_w(\Om_{R})}.
	    $$
	    From \eqref{eq-defn R} and \eqref{bounds-Du} we have
	    $$
	    \f{1}{R}\leq CK_0 \sim (|\mu|(\Om)^{\f{\sigma_0}{\gamma_1-1}}+|\Om|).
	    $$
	    These two inequalities imply
	    $$
	     \begin{aligned}
	     \Big\|\mathcal{M}(|Du|^{\f{q(\cdot)}{q_-}}\chi_{\Om_{2R}})\Big\|^{q_-}_{L^{q_-}_w(\Om_{R})}\leq &C((|\mu|(\Om)^{\f{\sigma_0}{\gamma_1-1}}
	     +|\Om|))^{nq_-}\Big(\int_{\Om_{2R}}|Du|^{1+\sigma_0} + 1 \ dx\Big)^{q_-} \\
	     &+2A_0^{q_-} \Big\|\Mu^{\f{1}{\pd-1}\f{q(\cdot)}{q_-}}\Big\|_{L^{q_-}_w(\Om_{R})}
	     \end{aligned}
	   $$
	   Therefore,
	   	  \begin{equation}\label{eq-localestimates}
	    \int_{\Om_R}|Du|^{q(x)}w(x)dx \leq C((|\mu|(\Om)^{\f{\sigma_0}{\gamma_1-1}}+|\Om|))^{n}\Big(\int_{\Om_{2R}}|Du|^{1+\sigma_0} + 1 \ dx\Big)
	    +2A_0 \int_{\Om_R}\Mu^{\f{q(x)}{\px-1}}w(x)dx.
	    \end{equation}
	    Since $\Om$ is bounded, there exists a family $\{\Om_R(x_i): x_i\in \Om, i=1,\ldots, N\}$ satisfying the following
	    \begin{enumerate}[(i)]
	    	\item $\Om \subset \bigcup_{i=1}^N \Om_R(x_i)$;
	    	\item there exists $C_0$ depending on $n$ only so that $\sum_{i=1}^N \chi_{\Om_{2R}(x_i)}\leq C_0$.
	    \end{enumerate}
	    We now apply \eqref{eq-localestimates} for each $\Om_{R}(x_i)$ and then sum up all the terms to conclude that
	    $$
	    \int_{\Om}|Du|^{q(x)}w(x)dx \leq C((|\mu|(\Om)^{\f{\sigma_0}{\gamma_1-1}}+|\Om|))^{n}\Big(\int_{\Om}|Du|^{1+\sigma_0} + 1 \ dx\Big)
	    +2A_0 \int_{\Om}\Mu^{\f{q(x)}{\px-1}}w(x)dx.
	    $$
	    This along with \eqref{bounds-Du} yields
	    $$
	    \int_{\Om}|Du|^{q(x)}w(x)dx \lesi ((|\mu|(\Om)^{\f{\sigma_0}{\gamma_1-1}}+|\Om|))^{n+1}
	    + \int_{\Om}\Mu^{\f{q(x)}{\px-1}}w(x)dx.
	    $$
	\end{proof}
	
	\begin{proof}
		[Proof of Corollary \ref{mainthm2}:] 
Applying  \eqref{eqs-mainthm} in Theorem \ref{mainthm1} for $w\equiv 1$ and $\f{nq(x)}{n-q(x)}$ taking place of $q(x)$ we have
		$$
		\int_{\Om} |Du|^{\f{nq(x)(\px-1)}{n-q(x)}}dx \leq C\left[(|\mu|(\Om)^{\f{\sigma_0}{\gamma_1-1}}+|\Om|)^{n+1} + \int_{\Om} |\Mu|^{\f{nq(x)}{n-q(x)}}dx\right].
		$$
		Then the conclusion of the corollary follows from the fact that $\mathbb{M}_1$ maps continuously from $L^{q(x)}(\Om)$ to $L^{\f{nq(x)}{n-q(x)}}(\Om)$. See for example \cite[Theorem 1.3]{CCF}.
	\end{proof}
		
		\bigskip
		
		\begin{proof}[Proof of Corollary \ref{mainthm3}:] 
			Applying  \eqref{eqs-mainthm} in Theorem \ref{mainthm1} for $q$ and $w^q$ taking place of $q(x)$ and $w$, respectively, we obtain \eqref{eq3-mainthm}.
				
				The inequality 	\eqref{eq3s-mainthm} follows from \eqref{eq3-mainthm} and the fact that
				$$
				\|\mathbb{M}_1f\|_{L^q_{w^q}(\Om)}\leq \|f\|_{L^r_{w^r}(\Om)},
				$$
			as long as $\f{1}{r}-\f{1}{q}=\f{1}{n}$ and $w^q\in A_{1+q/r'}$.
				\end{proof}

			\bigskip

		\bigskip
		
			
	\textbf{Acknowledgement.} The authors were supported by the research grant ARC DP140100649 from the Australian Research Council.

\end{document}